\documentclass[reqno]{amsart}
\usepackage{amsfonts}
\usepackage{amsfonts,amssymb,amsmath,xfrac}
\usepackage[latin1]{inputenc}
\usepackage{color}
\usepackage{multicol,caption,graphicx,float,lipsum}

\usepackage{graphicx}
\usepackage[active]{srcltx}
\usepackage{wrapfig}
\usepackage{multicap}
\usepackage{fullpage}
\usepackage{hyperref}
\usepackage{tikz}

\newtheorem{teo}{Theorem}[section]
\newtheorem{prop}[teo]{Proposition}
\newtheorem{lema}[teo]{Lemma}

\newtheorem{cor}[teo]{Corolary}

\theoremstyle{definition}
\newtheorem{defn}[teo]{Definition}

\theoremstyle{remark}
\newtheorem*{nota}{Remark}

\newcommand{\PP}{\mathbb{P}}
\newcommand{\EE}{\mathbb{E}}
\newcommand{\RR}{\mathbb{R}}
\newcommand{\QQ}{\mathbb{Q}}
\newcommand{\NN}{\mathbb{N}}
\newcommand{\ZZ}{\mathbb{Z}}

\newcommand{\MM}{\mathcal{M}}
\newcommand{\HH}{\mathcal{H}}
\newcommand{\charf}[1]{\mathbf{1}_{#1}}

\newcommand{\<}{\langle}
\renewcommand{\>}{\rangle}
\newcommand{\norm}[1]{\Vert #1 \Vert}
\newcommand{\floor}[1]{\lfloor #1 \rfloor}
\newcommand{\ind}[1]{\textbf{1}_{#1}}                           
\newcommand{\Ind}[1]{\textbf{1}{\lbrace{#1}\rbrace}}           
\newcommand{\mc}[1]{\mathcal{#1}}
\newcommand{\pfrac}[2]{\genfrac{}{}{}{1}{#1}{#2}}

\newcommand{\lrp}[1]{\left(#1\right)}                     
\newcommand{\lrc}[1]{\left[#1\right]}

\newcommand{\var}{\mathrm{Var}}

\title[Exclusion Process with Slow  Boundary]{Exclusion Process with Slow  Boundary}


\author{Rangel Baldasso}
\address{IMPA\\Estrada Dona Castorina, 110\\
Horto, Rio de Janeiro\\
Brasil}
\curraddr{}
\email{baldasso@impa.br}

\author{Ot\'avio Menezes}
\address{IMPA\\Estrada Dona Castorina, 110\\
Horto, Rio de Janeiro\\
Brasil}
\curraddr{}
\email{omenezes@impa.br}

\author{Adriana Neumann}
\address{UFRGS, Instituto de Matem\'atica, Campus do Vale, Av. Bento Gon\c calves, 9500. CEP 91509-900, Porto Alegre, Brasil}
\curraddr{}
\email{aneumann@mat.ufrgs.br}

\author{Rafael R. Souza}
\address{UFRGS, Instituto de Matem\'atica, Campus do Vale, Av. Bento Gon\c calves, 9500. CEP 91509-900, Porto Alegre, Brasil}
\curraddr{}
\email{rafars@mat.ufrgs.br}

\date{}

\begin{document}

\begin{abstract}

We study the hydrodynamic and the hydrostatic behavior of the Simple Symmetric Exclusion Process with  \emph{slow boundary}. The term \emph{slow boundary} means that particles can be born or die at the boundary sites, at a rate proportional to $N^{-\theta}$, where $\theta > 0$ and $N$ is the scaling parameter. In the bulk, the  particles exchange rate  is equal to $1$. In the  hydrostatic scenario, we obtain three different linear profiles, depending on the value of the parameter $\theta$; in the hydrodynamic scenario, we obtain that the time evolution of the spatial density of particles, in the diffusive scaling, is given by the weak solution of the heat equation, with boundary conditions that depend on $ \theta $. If $\theta \in(0,1)$, we get  Dirichlet boundary conditions, (which is the same behavior if $\theta=0$, see \cite{f});
if $\theta=1$,  we get Robin boundary conditions; and, if $\theta\in(1,\infty)$,  we get Neumann boundary conditions. 

\end{abstract}

\maketitle

\section{Introduction}

The problem we address here is a complete characterization of the hydrostatic and hydrodynamic scenario for the Symmetric Simple Exclusion Process (SSEP) with slow boundary. This particle system is a simple model of mass transfer between reservoirs of different densities. The SSEP is described by particles that move as independent symmetric random walks in $\{1, \dots, N-1\}$, under the exclusion rule, that says two particles cannot occupy the same site at the same time. By slow boundary, we mean that particles can enter the system at site 1 with rate $\sfrac{c\alpha}{N^{\theta}}$ or leave with rate $\sfrac{c(1-\alpha)}{N^{\theta}}$, while at the site $N-1$, particles exhibit a similar behavior, but with rates $\sfrac{c\beta}{N^{\theta}}$ for entrance and rate $\sfrac{c(1-\beta)}{N^{\theta}}$ for leaving the system. We consider the parameters $\alpha, \beta$ in $(0,1)$, $c>0$ and $\theta \geq 0$.

\begin{center}
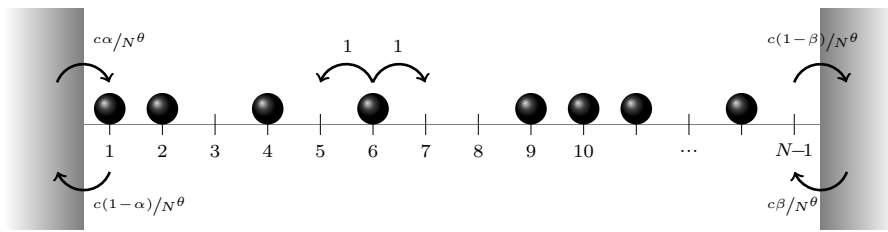
\begin{figure}[h!]
\begin{tikzpicture}[scale=0.7]
%

\draw[step=1cm,gray,very thin] (0,0) grid (15,0);

\foreach \y in {0,...,15}{
\draw[color=black,very thin] (\y,-0.2)--(\y,0.2);
}

\foreach \y in {1,...,2}{
\shade[ball color=black] (\y,0.3) circle (.3);
}

\foreach \y in {4,...,4}{
\shade[ball color=black] (\y,0.3) circle (.3);
}

\foreach \y in {6,...,6}{
\shade[ball color=black] (\y,0.3) circle (.3);
}
\foreach \y in {9,...,10}{
\shade[ball color=black] (\y,0.3) circle (.3);
}

\foreach \y in {11,...,11}{
\shade[ball color=black] (\y,0.3) circle (.3);
}
\foreach \y in {13,...,13}{
\shade[ball color=black] (\y,0.3) circle (.3);
}

\shade[left color=white,right color=gray] (-1, 2.2) rectangle  (0.5,-2);

\shade[right color=white,left color=gray] (14.5, 2.2) rectangle  (16,-2);


\draw [->,color=black, line width=1] (0,0.8) arc (160:20:15pt);
\draw (0.5,1.2) node [color=black, above right] {{$\scriptstyle{\sfrac{c\alpha}{N^{\theta}}}$}};

\draw [->,color=black, line width=1] (1,-0.9) arc (-20:-160:15pt);
\draw (0.5,-1.2) node [color=black, below right] {{$\scriptstyle{\sfrac{c(1\!-\!\alpha)}{N^{\theta}}}$}};

\draw [->,color=black, line width=1] (14,0.8) arc (160:20:15pt);
\draw (13.3,1.2) node [color=black, above right] {{$\scriptstyle{\sfrac{c(1\!-\!\beta)}{N^{\theta}}}$}};

\draw [->,color=black, line width=1] (15,-0.9) arc (-20:-160:15pt);
\draw (13.3,-1.2) node [color=black, below right] {{$\scriptstyle{\sfrac{c\beta}{N^{\theta}}}$}};

\draw [->,color=black, line width=1] (6,0.8) arc (160:20:15pt);
\draw (6.2,1.2) node [color=black, above right] {{$\scriptstyle{1}$}};
\draw [<-,color=black, line width=1] (5,0.8) arc (160:20:15pt);
\draw (5.8,1.2) node [color=black, above left] {$\scriptstyle{1}$};

\draw (1,-0.5) node [color=black] {$\scriptstyle{1}$};
\draw (2,-0.5) node [color=black] {$\scriptstyle{2}$};
\draw (3,-0.5) node [color=black] {$\scriptstyle{3}$};
\draw (4,-0.5) node [color=black] {$\scriptstyle{4}$};
\draw (5,-0.5) node [color=black] {$\scriptstyle{5}$};
\draw (6,-0.5) node [color=black] {$\scriptstyle{6}$};
\draw (7,-0.5) node [color=black] {$\scriptstyle{7}$};
\draw (8,-0.5) node [color=black] {$\scriptstyle{8}$};
\draw (9,-0.5) node [color=black] {$\scriptstyle{9}$};
\draw (10,-0.5) node [color=black] {$\scriptstyle{10}$};
\draw (12,-0.5) node [color=black] {$\scriptstyle{...}$};
\draw (14,-0.5) node [color=black] {$\scriptstyle{N\!-\!1}$};

\end{tikzpicture}
\caption{The model.}
\end{figure}
\end{center}

There are many works that consider similar models. In \cite{mariaeulalia1, me1} and \cite{mariaeulalia2}, the authors consider a model where deaths can only happen in an interval around the site 1 and entrance is allowed on an interval around the right boundary. Their model presents a current exchange between the two reservoirs and shares some similarities with the case $\theta=1$. Another case already studied in the literature (see \cite{f} and \cite{lmo}) is when the boundary is not slowed (that corresponds to $\theta=0$).
Finally, we mention the article \cite{b}. They consider an exclusion process in the discrete torus $\ZZ/N\ZZ$. At all bonds $(x,x+1)$, the exchange rate is equal to 1, except for a single special bond, which is easier to cross from left to right than from right to left. Their model presents a surprising phenomenon, which the authors call "battery effect". This effect consists in a current of particles through the system.


In the present paper, our goal is to understand the collective behavior of the time evolution of the SSEP with slow boundary. We study the limit of the time evolution of the spatial density of particles when the time and space are re-scaled in a suitable way. This scaling limit leads to the \emph{hydrodynamic limit}, which is usually characterized by the weak solution of some partial differential equation (PDE), called the \emph{hydrodynamic equation}.

The model analyzed here was motivated by a process considered in \cite{fgn1,fgn2,fgn3}, the SSEP with slow bond. This is a particle system in the discrete torus $\ZZ/N\ZZ$ where particles in neighbouring sites exchange positions with rate 1, except for only one edge, the \emph{slow bond}, where the exchange occurs with rate $N^{-\beta}$. The works \cite{fgn1,fgn2} exhibit different hydrodynamic behaviors, depending on the range of $\beta$. Nevertheless, the hydrodynamic equation in all cases is a parabolic equation, whose boundary conditions vary as the value of $\beta$ changes, exhibiting three different regimes. The intuitive idea is that, if we "open" the discrete torus, then the slow bond becomes a slow boundary. There, in the case $\beta=1$, the boundary conditions show a connection between the extremes $0$ and $1$. A natural question to ask is what is the role of the slow bond on this behavior.

Let us concentrate on the results of the present paper. As in \cite{fgn1,fgn2}, the model we study here has three different phases, depending on the value of $\theta$. The first case is $\theta \in [0,1)$. Here the hydrodynamic equation is the same as in the case $\theta =0$, the heat equation with Dirichlet boundary conditions
\begin{equation}\label{Dirichlet-intro}
\left\{
\begin{array}{l}
\partial_{t} \rho\,(t,u)  =  \partial_u^2 \rho\,(t,u) , \\
 \rho\, (t,0)  =  \alpha, \quad \rho \,(t,1)  =  \beta. \\
\end{array}
\right . 
\end{equation}
In the case $\theta >1$, the boundary is so slow that, as far as hydrodynamics is concerned, the system does not exchange mass with the reservoirs. That is to say, the hydrodynamic equation has Neumann boundary conditions:
\begin{equation}\label{Neumann-intro}
\left\{
\begin{array}{l}
\partial_{t} \rho\,(t,u)  =  \partial_u^2 \rho\,(t,u), \\
\partial_u \rho\, (t,0)  = \partial_u \rho\, (t,1)  = 0.
\end{array}
\right . 
\end{equation}
In the critical case $\theta=1$, the model presents a different macroscopic behavior. The hydrodynamic equation has Robin boundary conditions
\begin{equation}\label{Robin-intro}
\left\{
\begin{array}{l}
\partial_{t} \rho\,(t,u)  =  \partial_u^2 \rho\,(t,u), \\
\partial_{u} \rho\,(t,0)  = c(\rho\,(t,0)-\alpha), \\
\partial_{u} \rho\,(t,1) =  c(\beta-\rho\,(t,1)) .
\end{array}
\right . 
\end{equation}
These boundary conditions show that the rate of mass transfer through each site of the boundary depends only on what is happening in the corresponding site. To compare with a closely related result in the literature, the behavior is different from what is observed in the case $\beta=1$ of \cite{fgn1,fgn2}.


As expected, the main difficulty in the proof of the hydrodynamic limit is the characterization of the limit points. We overcome these difficulties with the use of Replacement Lemmas and Energy Estimates. However,  in general, the Bernoulli product measures are not invariant for the dynamics. In spite of that, we use these measures to prove such lemmas and estimates. The use of these measures instead of the invariant measures come with some cost, that vanishes in the limit. The case $\theta \in (0,1)$ has some further technical details, that we overcome with an idea from \cite{f}.

In order to understand the behavior of these invariant measures, we look at the profile associated to them. This is the \emph{hydrostatic limit}. Even though \cite{d} gives a characterization of these measures, we follow a different approach, that does not rely on their explicit computations.
As expected, the invariant measures are associated to the stationary solutions of the corresponding hydrodynamic equations. 
When $\theta>1$, the profile is constant equal to $(\alpha+\beta)/2$.
When $ \theta <1$, the profile is linear, with $\rho(0)=\alpha$ and $\rho(1)=\beta$.
\footnote{ This is the same profile found in \cite{lmo}. In this work, the authors study fluctuations for the case $\theta=0$.} 
The most interesting case is when $\theta =1$: the stationary solution $\rho:[0,1] \to \RR$ is linear with $\rho(0)=\alpha +(\beta - \alpha)/(c+2)$ and $\rho(1)=\beta - (\beta - \alpha)/(c+2)$. Here one can see the influence of the boundary 
As $c$ runs from $0$ to $\infty$, the linear profile in the hydrostatic limit interpolates between the constant profile ($\theta >1$) and the profile that goes from $\alpha$ to $\beta$ ($\theta <1$). 

The main challenge in the proof of hydrostatic limit is to get a good bound on the two point correlation function of the invariant measure. We use a coupling argument to relate the two point correlation function with the occupation time of a certain random walk.

The present work is divided as follows. In Section \ref{s2}, we introduce the notation and the main results. In Section \ref{s3}, we prove the hydrostatic behavior stated in Theorem \ref{t1}. The remaining of the paper is dedicated to the proof of Theorem \ref{t2}. In Section \ref{s4}, we prove tightness for any range of the parameter $\theta$. In Section \ref{s5}, we prove the {\em{Replacement Lemma}} and we establish the Energy Estimates, which are fundamental steps towards the proof. In Section \ref{s6}, we characterize the limit points as weak solutions of the corresponding partial differential equations. Finally, we establish uniqueness of weak solutions in Section \ref{s7}.

\section{Notation and main results}\label{s2}

\subsection{The Model}
Let $N\in\NN$ and $\theta >0$. We denote by $I_N$ the set $\{1,\ldots, N-1\}$, which will be called \emph{bulk}. The sites (points) of $ I_{N} $ will be denoted by  $ x $, $ y $ and $ z $, while the macroscopic variables (points of the interval $ [0,1] $) will be denoted by $ u $, $ v $ and $ w $. The microscopic state space will be denoted by $ \Omega_N:=\{0,1\}^{I_{N} }$ and its elements, called configurations, will be denoted by $ \eta $ and $ \xi $. Therefore, $ \eta(x) \in \{0,1\} $ represents the number of particles at site $ x $ for the configuration $ \eta $.

We split the generator of the SSEP with slow boundary, $L_{N}$, in three parts, one for the bulk dynamics, $L_{N,0}$, and two for the boundary dynamics, $L_{N,b}^\alpha$ and $L_{N,b}^\beta$, in the following way:
\begin{equation}\label{ger}
L_N =L_{N,0}+L_{N,b} ^\alpha+L_{N,b}^\beta.
\end{equation}
 The  generator of the  bulk dynamics  acts on functions $f:\Omega_N\to\RR$ as
\begin{equation*}
(L_{N,0}f)(\eta)=\sum_{x=1}^{N\!-\!2} \big[f(\eta^{x,x+1})-f(\eta)\big],
\end{equation*}
where $\eta^{x,x+1}$ is the configuration obtained from  $\eta$ by exchanging the variables $\eta(x)$ and $\eta(x+1)$: $\eta^{x,x+1}(x)=\eta(x+1)$,  $\eta^{x,x+1}(x+1)=\eta(x)$ and $\eta^{x,x+1}(y)=\eta(y)$, if $y\neq x,x+1$.
  The  generator at the left hand side of the boundary acts on  functions $f:\Omega_N\to\RR$ as
\begin{equation*}
\begin{split}
(L_{N,b}^\alpha f)(\eta)&:= cN^{-\theta}r_\alpha(\eta)\lrc{f(\eta^{1})-f(\eta)},
\end{split}
\end{equation*}
where\begin{equation}\label{ralpha}
r_\alpha(\eta)=\alpha (1-\eta(1))+(1-\alpha)\eta(1)
\end{equation} and the configuration $\eta^1$ flips the occupancy only in the site $1$: $\eta^1(1)=1-\eta(1)$ and $\eta^1(y)=\eta(y)$,  for all $y\in \{2,\dots,N-1\}$.
Similarly in the right hand side of the boundary we have the following   generator  that  acts on  functions $f:\Omega_N\to\RR$ as
\begin{equation*}
\begin{split}
(L_{N,b}^\beta f)(\eta)&:= c N^{-\theta}r_\beta(\eta)\lrc{f(\eta^{N\!-\!1})-f(\eta)},
\end{split}
\end{equation*}
with \begin{equation}\label{rbeta}
r_\beta(\eta)=\beta (1-\eta(N\!-\!1))+(1-\beta)\eta(N\!-\!1)
\end{equation}
and the configuration $\eta^{N-1}$ is similar to $\eta^1$: $\eta^{N-1}({N-1})=1-\eta({N-1})$ and $\eta^1(y)=\eta(y)$,  for all $y\in \{1,\dots,N-2\}$.

Denote by $\{\eta_t ;\; t\ge 0\}$  the Markov process on $\Omega_N$ associated to the generator $N^2L_N$.  Although $\{\eta_t ;\; t\ge 0\}$ depends on $N$, $ \alpha $, $ \beta $ and $ \theta $, 
these symbols are omitted to keep notation as simple as possible.
To avoid some technicalities,  we consider the trajectories of the Markov process $\eta_t $ with  $t\in [0,T]$, for some  $T>0$. Let $ D_{\Omega_N}[0,T] $ be the path space of
c\`adl\`ag time trajectories with values in $  \Omega_N $, the so called  Skorohod space.
  Given a
measure $\mu_N$ on $\Omega_N$, denote by $\PP_{\mu_N}$ the
probability measure on $ D_{\Omega_N}[0,T]$ induced by the
initial state $\mu_N$ and the Markov process $\eta_t$.
Expectation with respect to $\PP_{\mu_N}$ will be denoted by $\EE_{\mu_N}$.

\subsection{Hydrostatic Limit}\label{sub21}
A straightforward computation shows that, when $ \alpha=\beta$, the Bernoulli product measure with parameter $ \alpha $, defined by the fact that the random variables $ \{\eta(x), x \in I_N\} $ are independent and have distribution Bernoulli($\alpha$), is reversible for the dynamics. However, when $ \alpha \neq \beta $ this is not true.
For the general case what we can say about 
the invariant measures  is that they are associated to a linear profile depending on $\theta$, in the sense of Definition \ref{initialprofile}, as we state in the Theorem \ref{t1}.
\begin{defn}\label{initialprofile}
A sequence of probability measures $ \{\mu_{N} : N \geq 1 \} $ in the space $ \Omega_N $ is \emph{associated to the density profile} $ \gamma:[0,1] \rightarrow [0,1] $ if, for all $ \delta >0 $ and continuous function $ H: [0,1] \rightarrow \RR $ the following holds:
\begin{equation*}
\lim_{N\to\infty}
\mu_{N} \Big[ \eta:\, \Big| \frac {1}{N} \sum_{x = 1}^{N\!-\!1} H(\pfrac{x}{N})\, \eta(x)
- \int_{[0,1]} H(u)\, \gamma(u)\, du \Big| > \delta\Big] \;=\; 0.
\end{equation*}
\end{defn}

\begin{teo}[Hydrostatic Limit]\label{t1}
Let $\mu_N$ be the probability measure in $\Omega_N$ invariant for the Markov process with infinitesimal generator $N^2 L_N$, defined in \eqref{ger}. Then the sequence $\mu_N$ is associated to the profile $\overline{\rho}: [0,1]\to\RR$ given by
\begin{equation}\label{perfis_estacionarios}
\overline{\rho}(u)=
\begin{cases}
(\beta - \alpha)u+\alpha, & \mbox{if }\theta\in [0,1), \\
\frac{c(\beta -\alpha)}{2+c}u+\alpha + \frac{\beta - \alpha}{2+c}, & \mbox{if }\theta = 1, \\
\frac{\beta + \alpha}{2}, & \mbox{if }\theta \in(1,\infty),
\end{cases}
\end{equation}
for all $u\in [0,1]$.
\end{teo}
In Section \ref{s3} we prove this theorem.
Notice that  the profiles in \eqref{perfis_estacionarios} are precisely the stationary solutions of the hydrodynamic equations \eqref{Dirichlet-intro}, \eqref{Neumann-intro} and \eqref{Robin-intro}, respectively.

\subsection{Hydrodynamic Limit}\label{s3,5}

To state the Hydrodynamic Limit, see Theorem \ref{t2}, we need some notations and definitions, which we present in the following:
We denote by $ \< \cdot,\cdot \> $ the $L^{2}[0,1]$ inner product. When we consider $L^{2}[0,1]$ with respect to a measure $\mu$, we  denote by $ \<\cdot,\cdot \>_{\mu} $ its inner product. For an interval $ \mathcal{I} $ , a time $ T>0 $ and integers $ n $ and $ m $, we denote by $ C^{n,m}\left([0,T] \times \mathcal{I}\right) $ the set of functions defined on $ [0,T] \times \mathcal{I} $ that are $n$ times differentiable on the first variable and $m$ on the second one. An index on a function will always denote a variable, not a derivative. For example, $ H_{s}(u) $ means $ H(s,u) $. The derivative of $H \in C^{1,2}\left([0,T] \times \mathcal{I}\right)$ will be denoted by $\partial_s H$ (first variable) and $\partial_u H$ (second variable). We shall write $\Delta H$ for $\partial_{u}^2 H$.
We also have to consider the subset $C^{1,2}_0([0, T] \times [0,1]) $ of functions $H \in C^{1,2}([0,T] \times [0,1])$ such that $H_t(0)=0=H_t(1)$, for all $t\in [0,T]$. Finally, we denote by $C^{\infty}_{\textrm{k}}((0,1))$ the set of all smooth functions defined in  $(0,1)$ with compact support.

\begin{defn}

Let  $\HH^1(0,1)$ be the set of all locally integrable functions $g: (0,1)\to\RR$ such that
there exists a function $\partial_ug\in L^2(0,1)$ satisfying
\begin{equation*}
 \<\partial_u\varphi,g\>\,=\,-\<\varphi,\partial_ug\>,
\end{equation*}
for all $\varphi\in C^{\infty}_{\textrm{k}}((0,1))$.
For $g\in\HH^1(0,1)$,  define the norm
\begin{equation*}
 \Vert g\Vert_{\mc H^1(0,1)}\,:=\, \Big(\Vert g\Vert_{L^2(0,1)}^2+\Vert\partial_ug\Vert_{L^2(0,1)}^2\Big)^{1/2}.
\end{equation*}
\end{defn}

\begin{defn}
 The space $L^{2}(0,T;\HH^{1}(0,1))$ is the set of measurable functions $ f:[0,T] \to \HH^{1}(0,1) $ such that
\begin{displaymath}
\int_{0}^{T}||f_{t}||^{2}_{\HH^{1}(0,1)}\,dt < +\infty.
\end{displaymath}
\end{defn}

\begin{defn}[Hydrodynamic equation for $\theta\in[0,1)$]
Let $\gamma: [0,1]\to\RR$ be a measurable function. We say that a bounded function $ \rho : [0,T] \times [0,1] \rightarrow \RR $ is a weak solution of the heat equation with Dirichlet boundary conditions
\begin{equation}\label{dirichlet}
\left\{
\begin{array}{ll}
\partial_{t} \rho_t(u)  =  \Delta \rho_t(u),  &u\in(0,1),\, \,\,t\geq 0,  \\
 \rho_{t} (0)  =  \alpha, \,\, \rho_{t} (1)  = \beta,  & t \geq 0,\\
 \rho_0(u)  =  \gamma(u), & u\in[0,1],\\
\end{array}
\right . 
\end{equation}
if $\rho\in L^2(0,T;\HH^1(0,1))$ and  $ \rho $ satisfies
\begin{equation}\label{solucao_fraca_dirichlet}
\begin{split}
\< \rho_{t}, H_{t}\>-\< \gamma , H_{0}\> & = \int_{0}^{t} \< \rho_{s} , (\partial_{s}+\Delta)  H_{s} \> \,ds - \int_{0}^{t}\big\{\beta\partial_{u}H_{s}(1)-\alpha\partial_{u}H_{s}(0)\big\}\,ds,
\end{split}
\end{equation}
for all $ t \in {[0,T]} $ and $ H \in C^{1,2}_0([0, T] \times [0,1]) $.
\end{defn}

\begin{defn}[Hydrodynamic equation for $\theta=1$]
Let $\gamma: [0,1]\to\RR$ be a measurable function. We say that a bounded function $ \rho : [0,T] \times [0,1] \rightarrow \RR $ is a weak solution of the heat equation with Robin boundary conditions

\begin{equation}\label{robin}
\left\{
\begin{array}{ll}
\partial_{t} \rho (u)  =  \Delta \rho_t(u),  &u\in(0,1),\, \,\,t \geq 0,   \\
\partial_{u} \rho_{t} (0)  = c(\rho_{t} (0)-\alpha), & \, t\geq 0,\\
 \partial_{u} \rho_{t} (1) = c(\beta-\rho_{t} (1)), & \, t \geq 0,\\
\rho_0(u)  =  \gamma(u),  & u\in[0,1],  \\
\end{array}
\right .
\end{equation}
if $\rho\in L^2(0,T;\HH^1(0,1))$ and $ \rho$ satisfies
\begin{equation}\label{solucao_fraca_robin}
\begin{split}
\< \rho_{t}, H_{t}\>-\< \gamma , H_{0}\> & = \int_{0}^{t} \< \rho_{s} , (\partial_{s}+\Delta)  H_{s} \> \,ds\\  
& + \int_{0}^{t}\big\{\rho_s(0)\partial_{u}H_{s}(0)-\rho_{s}(1)\partial_{u}H_{s}(1)\big\}\,ds \\
& + c\int_{0}^{t}\Big\{H_{s}(0)\big(\alpha-\rho_{s}(0)\big) + H_{s}(1)\big(\beta-\rho_{s}(1)\big)\Big\}\,ds,
\end{split}
\end{equation}
for all $ t \in {[0,T]} $ and $ H \in C^{1,2}([0, T] \times [0,1]) $.
\end{defn}

\begin{defn}[Hydrodynamic equation for $\theta>1$]
Let $\gamma: [0,1]\to\RR$ be a measurable function. We say that a bounded function $ \rho : [0,T] \times [0,1] \rightarrow \RR $ is a weak solution of the heat equation with Neumann boundary conditions

\begin{equation}\label{neumann}
\left\{
\begin{array}{ll}
\partial_{t} \rho(u)  =  \Delta \rho_t(u) , & u\in(0,1),\, \,\,t \geq 0,  \\
\partial_{u} \rho_{t} (0) =0, \,\,\partial_{u} \rho_{t} (1)  =  0, &  t \geq 0,\\
\rho_0(u)  =  \gamma(u) ,& u\in[0,1],  \\
\end{array}
\right .
\end{equation}
if $\rho\in L^2(0,T;\HH^1(0,1))$ and $ \rho $ satisfies
\begin{equation}\label{solucao_fraca_neumann}
\begin{split}
\< \rho_{t}, H_{t}\>-\< \gamma , H_{0}\> & = \int_{0}^{t} \< \rho_{s} , (\partial_{s}+\Delta)  H_{s} \> \,ds\\  
& - \int_{0}^{t}\big\{\rho_s(1)\partial_{u}H_{s}(1)-\rho_{s}(0)\partial_{u}H_{s}(0)\big\}\,ds,
\end{split}
\end{equation}
for all $ t \in {[0,T]} $ and $ H \in C^{1,2}([0, T] \times [0,1]) $.
\end{defn}

\begin{nota}
As usual, to have an idea as we obtain the weak formulation of the heat equation with the boundary conditions  showed in 
one should multiply both sides of $\partial_{t} \rho(u)  =  \Delta \rho_t(u) $  by a test function $H$, then integrate both in space and time and finally, perform twice a formal integration by parts in space and one in time. Then, applying the  respective boundary conditions we are
lead to the corresponding integral equation. This  shows  that any strong solution is a weak solution of the respective equation.
\end{nota}
  Now we are able to enunciate our second statement:

\begin{teo}[Hydrodynamic Limit]\label{t2}
Let $ \gamma : [0,1] \rightarrow [0,1 ]$ be a measurable function and $ \{\mu_{N}: N \geq 1 \} $ a sequence of probability measures in $\Omega_N$ associated to $ \gamma $, in the sense of Definition \ref{initialprofile}. For each $ t \in [0,T] $, $ \delta >0 $ and for all functions $ H \in C[0,1] $, we have that
\begin{equation*}
\lim_{ N \rightarrow +\infty }
\PP_{\mu_{N}} \Big[\eta_{\cdot}: \Big\vert \frac{1}{N} \sum_{x \in I_{N}}
\eta_{t}(x) \,H(\pfrac{x}{N})\, - \int_{[0,1]} H(u)\, \rho(t,u)\, du\, \Big\vert
> \delta \Big] = 0,
\end{equation*}
holds, where  $\rho(t,\cdot)$ is the unique weak solution of the equation \eqref{dirichlet} if $\theta\in[0,1)$, or \eqref{robin} if $\theta = 1$, or \eqref{neumann} if $\theta\in(1,\infty)$.
\end{teo}

  In order to prove the theorem above, we begin by defining the empirical measure associated to a configuration $ \eta $ as
\begin{displaymath}
\pi^N(\eta,du)=\frac{1}{N} \sum_{x \in I_{N}}\eta(x)\delta_{\frac{x}{N}}(du),
\end{displaymath}
where $ \delta_{u} $ denotes the Dirac mass at $ u $. Thus, given a Markov process in $ D_{\Omega_N}[0,T] $, we can consider the empirical process $ \pi^N_{.}(du)=\pi^N(\eta_{.},du) $ in $ D_{\MM}[0,T] $, where $ \MM $ is the set of positive measures on $ [0,1] $  with total mass bounded by 1 endowed with the weak convergence topology. Observe that $ \pi_{.}^{N} $ is also a Markov process.
For an integrable function $ H: [0,1] \rightarrow \RR $, $\< \pi^{N}_{t}, H \>$ denotes the integral of $ H $ with respect to the measure $ \pi^{N}_{t} $:
\begin{equation*}
\< \pi^{N}_{t}, H \> = \frac{1}{N} \sum_{x \in  I_{N}} \eta_{t}(x)\,H (\pfrac{x}{N})
.
\end{equation*}
When $ \pi_{t} $ has a density $ \rho_t $ ($ \pi(t,du) = \rho(t,u) du $), we will write $ \< \rho_{t}, H \> $ instead of $ \< \pi_{t}, H \> $.

  Fix $ T>0 $ and a value of $ \theta\geq 0 $. Given a probability measure $ \mu_{N} $ on $ \Omega_N $, consider the Markov process $ \pi_{.}^{N} $ in $ D_{\MM}[0,T] $ associated to the process $ \eta_{.}^{N} $ in $ D_{\Omega_N}[0,T] $, which has $ \mu_{N} $ as initial distribution. Denote by $ \QQ_{N} $ the probability measure on $ D_{\MM}[0,T] $ induced by $ \pi_{.}^{N} $.
 Fix an initial continuous profile $ \gamma : [0,1] \rightarrow [0,1] $ and consider a sequence $ \{ \mu_{N} : N \geq 1 \} $ of measures on $ \Omega_N $ associated to $ \gamma $. Let $ \QQ_{*} $ be the probability measure on $ D_{\MM}[0,T] $ which gives mass 1 to the path $ \pi(t,du) = \rho (t,u)du $, where $ \rho(t,\cdot) $ is the weak solution of the corresponding PDE (keep in mind that the PDE depends on the value of $\theta$).

Now, we are in position to state our main result, a slight generalization of Theorem \ref{t2}.

\begin{prop}\label{s15}
The sequence of probabilities $ \{ \QQ_{N}\}_{N \in \NN} $ converges weakly to $ \QQ_{*} $, when $ N \rightarrow +\infty $.
\end{prop}

Observe that  Theorem \ref{t2} is a corollary of the proposition above.
  The proof of this proposition is divided into three steps. In  Section \ref{s4}, we prove that the sequence  $ \{ \QQ_{N}\}_{N \in \NN} $ is tight. In Section \ref{s6}, we prove that the limit points are absolutely continuous and the (Lebesgue) densities $ \rho_{t} $ are weak solutions of the corresponding hydrodynamic equation. In Section \ref{s7}, we prove uniqueness of weak solutions for each PDE.

\section{Proof of Theorem \ref{t1}}\label{s3}

In this section, we prove Theorem \ref{t1}. The proof is based on the approach of \cite{lmo}. The strategy is to compare $\eta(x)$ with its mean
\begin{equation}\label{mean}
\rho^N(x):=\EE_{\mu_N}[\eta(x)]
\end{equation} 
and $\rho^N$ with $\bar{\rho}$, defined in \eqref{perfis_estacionarios}.
The comparison between $\rho^{N}$ and $\bar{\rho}$ is stated as Lemma \ref{mean_occupation}. To compare the two remaining quantities, an important step is to control the decay of the two point correlation function
\begin{equation}\label{cov}
\varphi^N(x,y):=\EE_{\mu_N}[\eta(x)\eta(y)]-\EE_{\mu_N}[\eta(x)]\EE_{\mu_N}[\eta(y)].
\end{equation}
This is the content of Lemma \ref{covariancias}. The idea in the proof of this lemma is to interpret the quantities $\varphi^N(x,y)$ by means of a continuous time random walk. Once it is done, a coupling of random walks will provide the bounds on these quantities.
Now we state the two lemmas. Their proofs are postponed to the end of this section.
The first lemma computes explicitly the mean function $\rho^{N}$.
\begin{lema}\label{mean_occupation}
	The mean $\rho^N(x)$ is given by
	\begin{equation}\label{eq_mean_occupation}
	\rho^N(x)=a_Nx+b_N,\mbox{ for all  } x\in I_N,
	\end{equation}
	where 
	\begin{equation}\label{linear_coef}
	a_N=\frac{c(\beta -\alpha)}{2N^\theta+c(N\!-\!2)}
	\quad\quad\mbox{and}\quad\quad
	b_N=\alpha + a_N\left(\frac{N^\theta}{c} -1\right).
	\end{equation}
	
	In particular,
	\begin{equation}\label{eq:mean_limit}
	\lim_{N\to\infty}\Big(\max_{x\in I_N}\big|\rho^N(x)-\overline{\rho}(\pfrac{x}{N})\big|\Big)=0.
	\end{equation}
\end{lema}
The second lemma gives the asymptotic behavior of the two point correlation function.
\begin{lema}\label{covariancias}
	The two point correlation function $ \varphi^N(x,y)$ satisfies
	\begin{equation}\label{cov_3}
	\max_{0< x<y< N}\big|\varphi^N(x,y)\big|\leq\frac{C}{N^\theta+N},
	\end{equation}
	for some positive constant $C>0$.
\end{lema}
Now we show how to combine these lemmas to prove Theorem \ref{t1}.
\begin{proof}[Proof of Theorem \ref{t1}]
	Begin by Markov's inequality and triangular inequality to obtain
	\begin{equation*}
	\begin{split}
	& \mu_{N} \Big[ \eta:\, \Big| \pfrac {1}{N} \sum_{x = 1}^{N\!-\!1} H(\pfrac{x}{N})\, \eta(x)
	- \int_{[0,1]} H(u)\, \overline{\rho}(u) du \Big| > \delta\Big] \\
	\leq &\, \,\delta^{-1}\,\EE_{\mu_N}\left[\Big|
	\frac{1}{N}\sum_{x=1}^{N\!-\!1}H(\pfrac{x}{N})\left(
	\eta(x)-\rho^N(x)
	\right)\Big|
	\right]\\
	& + \delta^{-1}\Big|
	\frac{1}{N}\sum_{x=1}^{N\!-\!1}H(\pfrac{x}{N})\rho^N(x)-\int_0^1 H(u)\overline{\rho}(u)\,du\Big|.
	\end{split}
	\end{equation*}
Equation \eqref{eq:mean_limit} implies that the second term above converges to zero as $N\to\infty$ (a way to prove this is to approximate the integral by a Riemann sum).  For the first term, use Cauchy-Schwarz inequality and the fact that $|\eta(x)-\rho^N(x)|\leq 1$ to bound
	\begin{equation*}
	\begin{split}
	\EE_{\mu_N}\left[\Big|
	\frac{1}{N}\sum_{x=1}^{N\!-\!1}H(\pfrac{x}{N})\left(
	\eta(x)-\rho^N(x)
	\right)\Big|
	\right] & \leq \EE_{\mu_N}\left[\left(
	\frac{1}{N}\sum_{x=1}^{N\!-\!1}H(\pfrac{x}{N})\left(
	\eta(x)-\rho^N(x)
	\right)\right)^2
	\right]^{1/2}\\
	& \leq \norm{H}_{\infty}\left(
	\frac{1}{N} + 2\max_{1\leq x<y\leq N\!-\!1}|\varphi^N(x,y)|
	\right)^{1/2},
	\end{split}
	\end{equation*}
	which converges to zero by Lemma \ref{covariancias}.
\end{proof}

It remains to prove Lemmas \ref{mean_occupation} and \ref{covariancias}. We begin by the proof of Lemma \ref{mean_occupation}.

\begin{proof}[Proof of Lemma \ref{mean_occupation}] Fix $x\in I_N$ and let $f_x: \Omega_N\to\RR$ given by $f_x(\eta) = \eta(x)$. 
	Applying the infinitesimal generator  on the function $f_x$ we get
	\begin{equation*}
	L_{N}f_x(\eta)=
	\begin{cases}
	\big[\eta(x+1)-\eta(x)\big]+\big[\eta(x-1)-\eta(x)], & \mbox{if }x\in\{2,\dots,N\!-\!2\}, \\
	\big[\eta(2)-\eta(1)\big]+cN^{-\theta}\big[\alpha -\eta(1)\big], & \mbox{if } x=1, \\
	cN^{-\theta}\big[\beta -\eta(N\!-\!1)\big]+\big[\eta(N\!-\!2)-\eta(N\!-\!1)\big], & \mbox{if } x=N\!-\!1. \\
	\end{cases}
	\end{equation*}
	Since $\mu_N$ is an invariant measure, $\EE_{\mu_N}[L_N f_x]=0$. Combining this fact with the system of equations above it is easy to conclude that $\rho^N(x)$ satisfy the following recurrence relations:
	\begin{equation*}
	\begin{cases}
	0 =[\rho^N(x+1)-\rho^N(x)]+[\rho^N(x-1)-\rho^N(x)],\quad \quad\mbox{if } \;\;\;x\in\{2,\dots,N\!-\!2\},\\
	0 =[\rho^N(2)-\rho^N(1)]+cN^{-\theta}[\alpha - \rho^N(1)], \\
	0 =cN^{-\theta}[\beta - \rho^N(N\!-\!1)]+[\rho^N(N\!-\!2)-\rho^N(N\!-\!1)].
	\end{cases}
	\end{equation*}
	The first equation above implies that $ \rho_N $ is a linear function. A direct verification shows that \eqref{eq_mean_occupation} and \eqref{linear_coef}  solve the recurrence, concluding the proof.
\end{proof}

We now turn to the proof of Lemma \ref{covariancias}. The strategy of the proof is the same as in the last lemma. The difficulty is that the system of equations we obtain for the functions $\varphi^{N}$ is not easily solved. A comparison with continuous time random walks allows us to get the desired bounds.

\begin{proof}[Proof of Lemma \ref{covariancias}]
	Let $1\leq x< y\leq N\!-\!1$ and $f_{xy}:\Omega_N\to\RR$ defined by 
	\begin{equation*}
	f_{xy}(\eta):=[\eta(x)-\rho^N(x)][\eta(y)-\rho^N(y)].
	\end{equation*} 
	Compute the generator $L_N$ applied on the function $f_{xy}$:
	\begin{equation}\label{cov_1}
	\begin{split}
	L_{N,0}f_{xy}(\eta) & =[\eta(x-1)-\rho^N(x)][\eta(y)-\rho^N(y)]\Ind{x>1}\\
	& + [\eta(x+1)-\rho^N(x)][\eta(y)-\rho^N(y)]\Ind{y\neq x+1}\\
	& + [\eta(x)-\rho^N(x)][\eta(y-1)-\rho^N(y)]\Ind{y\neq x+1}\\
	& + [\eta(x+1)-\rho^N(x)][\eta(x)-\rho^N(x+1)]\Ind{y= x+1}\\
	& + [\eta(x)-\rho^N(x)][\eta(y+1)-\rho^N(y)]\Ind{y< N\!-\!1}\\
	& - [\eta(x)-\rho^N(x)][\eta(y)-\rho^N(y)][2+\Ind{x>1}+\Ind{y< N\!-\!1}],\\
	\end{split}
	\end{equation}
	
	\begin{equation}\label{cov_2}
	\begin{split}
	L_{N,b}^\alpha f_{xy}(\eta)  =  &cN^{-\theta}[\alpha(1-\eta(1))+(1-\alpha)\eta(1)]\cdot[1-2\eta(1)]\cdot[\eta(y)-\rho^N(y)]\Ind{x=1}\quad\mbox{and}\\
	L_{N,b}^\beta f_{xy}(\eta) =&cN^{-\theta}[\beta(1-\eta(N\!-\!1))+(1-\beta)\eta(N\!-\!1)]\cdot[1-2\eta(N\!-\!1)][\eta(x)-\rho^N(x)]\Ind{y=N\!-\!1}.
	\end{split}
	\end{equation}
	Since the measure $\mu_N$ is invariant, we have $\EE_{\mu_N}[L_N f_{xy}]=0$. Combining this formula with equations \eqref{cov_1} and \eqref{cov_2}, we get a linear system of equations relating the covariances $\varphi^N(x,y)$. This linear system can be written in matrix form as $ A\varphi^{N} (x,y)=b(x,y) $ for adequate choices of the matrix $ A $ and the vector $ b $. The matrix $ A $ can be identified  as the generator of a random walk, constructed in the following.

Consider the triangle
\begin{equation*}
V:=\{
(x,y)\in\ZZ^2: 0\leq x <y\leq N
\},
\end{equation*}
and its boundary
\begin{equation*}
\partial V:=\{
(x,y)\in V: x=0 \text{ or } y=N
\}.
\end{equation*}
For $u,v\in V$, define the \emph{conductances}
\begin{equation*}
c_{\theta}(u,v)=
\begin{cases}
1 & \text{ if }  |u-v|=1, u\notin \partial V, v\notin \partial V,\\
cN^{-\theta} & \text{ if }|u-v|=1, u\notin \partial V, v\in\partial V, \\
0 & \text{ otherwise.}
\end{cases}
\end{equation*}
In the Figure 2 we have some copies of $V$, in each one the dashed edges corresponding to the ones where the conductances are equal to $N^{-\theta} $.
For $f: V\to\RR$, define
\begin{equation}\label{antheta}
A^{\theta}_Nf(u)=\sum_{v\in V}c_\theta(u,v)[f(v)-f(u)].
\end{equation}

	Notice that $\varphi^{N}$ is defined in $V \setminus \partial V$. If this function is extended to $V$ by declaring it to vanish in $\partial V$, then
	\begin{equation}\label{greenfunction}
	A_N^\theta \varphi^N(x,y) = a_N^2\Ind{y=x+1},
	\end{equation}
	with $a_N$ given by \eqref{linear_coef}.
	In the case $\theta = 0$, it is known that equation \eqref{greenfunction} has an unique solution given by
	\begin{equation}\label{thetazero}
	\varphi^N(x,y) =- \frac{a_N^2}{N\!-\!1}x(N-y),
	\end{equation}
	and \eqref{cov_3} easily follows. Then,	
	it remains to estimate the covariances when $ \theta >0 $. The strategy is to use the representation \eqref{greenfunction} and prove that to get an estimate for $ \varphi^N(x,y) $, one needs to obtain estimates on some occupation times of the random walk with generator $ A^{\theta}_N $, denoted by $ X^\theta $. We will define a coupling between $ X^\theta $ and $ X^0 $, compare the occupation times for these chains and use equation \eqref{thetazero} to conclude.

	Let $\{X^{\theta}_u(s),\,s\geq 0\}$ be the random walk with infinitesimal generator $A_N^\theta$ starting from $u=(x,y)\in V$. Notice that $X^0_u$ is a simple symmetric random walk on $V$, absorbed in $\partial V$. Denote the diagonal of $V$ by $D:=\{(x,x+1)\in V\}$. Then
	\begin{equation*}
	\varphi^N(X^{\theta}_u(t)) - \varphi^N(x,y) - a_N^2\int_0^t\Ind{X^{\theta}_u(s)\in D}\,ds
	\end{equation*}
	is a martingale with respect to the natural filtration $\mathcal{F}_t :=\sigma \{X^{\theta}_u(s): 0\leq s\leq t \}$. Taking expectations in the expression above, we get
	\begin{equation*}
	\varphi^N(x,y)=\EE_u[\varphi^N(X^{\theta}_u(t))]-a^2_N\cdot\EE_u\left[\int_0^t \Ind{X^{\theta}_u(s) \in D}\,ds\right].
	\end{equation*}
	Since the random walk will almost surely be absorbed at $\partial V$, where $\varphi^N$ vanishes, taking the limit $t\to\infty$ in the equation above we obtain 
		\begin{equation}\label{occupationtime}
	\varphi^N(x,y) = -a^2_N \cdot T^{\theta}_u,
	\end{equation}
		where $T^{\theta}_u:=\EE\left[\int_0^\infty \Ind{X^{\theta}_u(s) \in D}\,ds\right]$ and $u=(x,y)\in V$.
Thus, we need to estimate $T^{\theta}_u$, the total time spent by the random walk 
	$\{ X^{\theta}_u(s);\,s\geq 0\}$ on the diagonal $D$. For 
	$\theta=0$, 
using  equations \eqref{thetazero} and \eqref{occupationtime}, we have
	\begin{equation}\label{t0u}
	T^0_u = \frac{x(N-y)}{N\!-\!1}.
	\end{equation}
	We will compare $T^{0}_{u}$ and $T^{\theta}_{u}$ using a coupling.
	
Let us give an informal description of the coupling. We will construct a new process taking values in $V\times \NN$ such that its projection in $V$ has the same law of the process $X^\theta_u$. Write $u_1=u \in V$. The construction starts with a realization of $X^0_{u_1}$ in $V\times\{1\}$. When the walker tries to jump to the absorbing set $\partial V$, flip an independent coin with probability $cN^{-\theta}$ of heads. If it comes up heads, the walker jumps to $\partial V$ and is absorbed.
Otherwise, if the process is at the point $(u_2,1)$, the walker jumps to  $(u_2,2)$ and starts as an independent copy of $X^0_{u_2}$ in $V\times\{2\}$. When it tries to jump to $\partial V$ again, flip another independent coin and repeat the procedure until the random walk gets absorbed, see Figure 2.

Now, let us construct the coupling rigorously and use it to obtain bounds on $T^{\theta}_u$. Define a new Markov process $\{Z(s),\, s\geq 0\}$ with state space $ V \times \NN $ starting at the point $(u_1,1) \in V \times \NN$.
	Let $\{Y_{n}\}_{n \in \NN}$ be a sequence of independent and identically distributed Bernoulli($N^{-\theta}$) random variables and $Y=\inf\{k : Y_k=1\}$. Set, inductively $ \tau_n=\inf\{s: (X^0_{u_{n}} (s),n)\in \partial V \times \{n\} \} $, with $u_{n}$ defined by $(u_{n},n-1)=(X^{0}_{u_{n-1}}(\tau_{n-1}^{-}), n-1)$, and $\xi_{n}=\sum_{k=1}^{n}\tau_{k}$.
	
	\begin{figure}
	\begin{center}
\includegraphics[width=10cm]{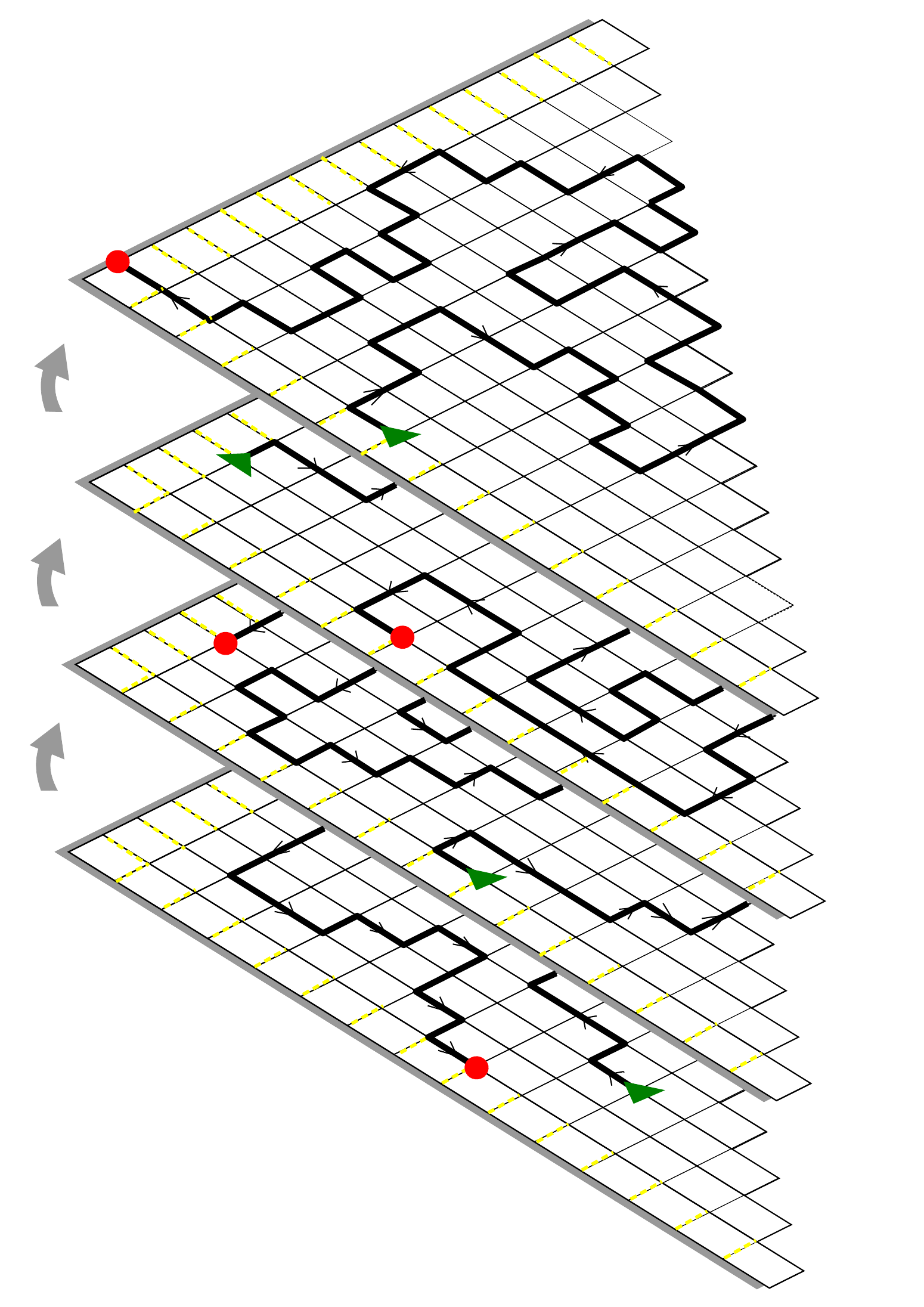}
\end{center}	
\caption{Coupling between two random walks, one absorbed at rate 1 and  the other at rate $N^{-\theta}$}
	\end{figure}

 	 Define  $Z(s)$  as
	\begin{equation*}
	Z(s)=
	\begin{cases}
	(X^0_{u_1}(s),1), & \mbox{if } s < \tau_{1}, \\
	\big(X^0_{u_{n}}(s-\xi_{_{n\!-\!1}}),n\big), & \mbox{if } \xi_{_{n\!-\!1}} \leq s < \xi_{n}, \,\,\, 2 \leq n \leq Y, \\
	(X^0_{u_{_{Y}}}(\tau_{_Y}),Y), & \mbox{if } s \geq \xi_{_{Y}}.
	\end{cases}
	\end{equation*}	
There are some important points to note about the random walk $\{Z(s),\,s\geq 0\}$. First, observe that the projection on the first coordinate has the law of $X^{\theta}$. Besides, the time $Z(s)$ spends in the diagonal in the level $n$, $D \times \{n\}$, is the same as the time spent in the diagonal by the random walk $X^0_{u_n}$ and these times have expectation bounded by $CN$ (for some universal constant $C$) regardless of the initial point, as observed in \eqref{t0u}. Another corollary of \eqref{t0u} is that if we begin the process $X^0_{u_n}$ in a point $u_n$ where $\partial V $ can be hit in only one jump, we can bound  $T^{0}_{u_n}$ by some constant $C$. This is the case when $n\geq 2$. Moreover, it is clear that the distribution of $Y$ is Geometric($N^{-\theta}$).
	
	For $n \geq 1$, define the random variables 
	\begin{equation*}
	D^{(n)}:=\int_0^\infty \Ind{Z(s)\in D\times \{n\}}\,ds.
	\end{equation*}
	Using \eqref{t0u}, we can bound $\EE[D^{(n)}]$ by a positive constant $C$, if $n \geq 2 $ and by $CN $ if $n=1$, where $C$ does not depend on the starting point $u_1$. This implies that
	\begin{equation*}
	\begin{split}
	T^{\theta}_{u}=\EE[D^{(1)}+\cdots +D^{(Y)}] &=\sum_{n\geq 1}\EE[(D^{(1)}+\cdots +D^{(n)})\Ind{Y = n}]\\
	&\leq \sum_{n\geq 1}CN\PP(Y = n)+\sum_{n\geq 1}C(n\!-\!1)\PP(Y = n)\\
	& \leq C(N+N^\theta).
	\end{split}
	\end{equation*}
	Equation \eqref{occupationtime} implies that $|\varphi ^N(x,y)|\leq C(N+N^{\theta })a_N^2$. Equation \eqref{linear_coef} concludes the proof.
	
\end{proof}

\section{Tightness}\label{s4}

In this section we prove the following proposition:
\begin{prop}\label{s06}
The sequence of probabilities $ \{\QQ_{N}\}_{N \in \NN} $ is tight in the Skorohod topology of $ D_{\MM}[0,T] $.
\end{prop}

The proof is divided into 2 cases: $ \theta \geq 1 $ and $ \theta\in[0,1) $. We begin with $\theta \geq 1$.
\begin{proof}[Tightness for $ \theta \geq 1 $]
  In order to prove the assertion, it is enough to prove that, for all $\varepsilon >0$ and for all $H$ in a dense subset of $ C[0,1] $, 
\begin{equation}\label{tight2}
\lim_{\delta \rightarrow 0} \, \limsup_{N \rightarrow \infty} \, \PP_{\mu_N}\left[ \sup_{|t-s| \leq \delta} | \<\pi^{N}_{t} ,H\> -  \<\pi^{N}_{s} ,H\>| > \varepsilon \right] =  0.
\end{equation}
We take $H\in C^2[0,1]$. 
By Dynkin's formula we know that
\begin{equation}\label{M}
M^{N}_{t}(H):=\<\pi^{N}_{t}, H\>- \<\pi^{N}_{0},
H\>-\int_{0}^{t}N^{2}L_{N}\<\pi^{N}_{s},H\>\,ds
\end{equation}
is a martingale with respect to the natural filtration $ \mathcal{F}_{t}:=\sigma(\eta_{s}: s \leq t)$. By the previous expression, \eqref{tight2} holds if we prove that
\begin{equation}\label{tight5}
\lim_{\delta \rightarrow 0} \, \limsup_{N \rightarrow \infty} \, \EE_{\mu_{N}}\left[ \sup_{|t-s| \leq \delta} \big| M^{N}_{t}(H) -  M^{N}_{s}(H) \big| \right]=  0 \, ,
\end{equation}
and
\begin{equation}\label{tight6}
\lim_{\delta \rightarrow 0} \, \limsup_{N \rightarrow \infty} \, \EE_{\mu_{N}}\left[ \sup_{| t-s| \leq \delta} \Big|\int_{s}^{t}N^{2}L_{N}\<\pi^{N}_{r},H\>\,dr \Big| \right] =  0 \, .
\end{equation}
 In order to verify \eqref{tight5}, we will use the quadratic variation of  $M^{N}_{t}(H) $, which we denote by $ \<M^{N}_{t}(H)\> $. By Doob's inequality and the fact that $\{ M^N_t(H)^2 - \< M^N_t(H) \> \}_{t\geq 0}$ is a mean zero martingale, we have
\begin{equation*}
\begin{split}
\EE_{\mu_{N}} \left[ \sup_{|t-s| \leq \delta} \big| M^{N}_{t}(H) -  M^{N}_{s}(H)\big| \right] & \leq  \,4 \, \EE_{\mu_{N}} \left[ \<M^{N}_{T}(H)\> \right]^{\sfrac{1}{2}}.
\end{split}
\end{equation*}
We now prove that the quadratic variation $ \<M^{N}_{t}(H)\> $ converges to zero uniformly in $ t \in [0,T] $, when $ N \to \infty $.
We consider separately the contributions due to $ L_{N,0} $ and to 
$$ L_{N,b}:=L_{N,b}^\alpha+L_{N,b}^\beta .$$ First, we compute
\begin{equation*}
\begin{split}
N^2\big[L_{N,0}\<\pi^{N}_{s},H\>^{2}-& 2\<\pi^{N}_{s},H\>L_{N,0}\<\pi^{N}_{s},H\>\big]\\
&= \frac{1}{N^2}  \sum_{x=1}^{N\!-\!2}(\eta_{s}(x)-\eta_{s}(x+1))^{2}\left(N\Big[H(\pfrac{x+1}{N})-H(\pfrac{x}{N})\Big]\right)^{2} .  
\end{split}
\end{equation*}
By the mean value theorem, the term in parentheses is bounded by $\Vert (H')^2 \Vert_{\infty} $. Combining this bound with $(\eta_{s}(x)-\eta_{s}(x+1))^{2} \leq 1  $, we can prove that
\begin{equation*}
\lim_{N \to \infty}\int_{0}^{t} N^{2} [ L_{N,0} \< \pi^{N}_{s}, H \> ^2 - 2\< \pi^{N}_{s}, H
\> L_{N,0} \< \pi^{N}_{s}, H \>]\, ds = 0,
\end{equation*}
uniformly in $ t \in [0,T] $.
  The next step consists in verifying that
\begin{equation}\label{L_N,b}
\lim_{N \to \infty}\int_{0}^{t} N^2 [ L_{N,b} \< \pi^{N}_{s}, H \>^{2} - 2\< \pi^{N}_{s}, H\> L_{N,b} \< \pi^{N}_{s}, H \>]\, ds = 0,
\end{equation}
uniformly in $ t \in [0,T] $. We  compute
\begin{align*}
N^2 [ L_{N,b}\<\pi^{N}_{s},H\>^2 -  2\<\pi^{N}_{s},H\>L_{N,b}\<\pi^{N}_{s},H\> ]
 = \;& c\,N^{-\theta}\left[r_\alpha(\eta)H(\pfrac{1}{N})^2 
+r_\beta(\eta)H(\pfrac{N\!-\!1}{N})^{2}\right],
\end{align*}
with $r_\alpha$ and $r_\beta$ given by \eqref{ralpha} and \eqref{rbeta}.
 The absolute value of the last expression is bounded by $ 2c\Vert H^{2} \Vert_{\infty}N^{-\theta} $, and therefore we have \eqref{L_N,b}. This concludes the verification of \eqref{tight5}.

In order to verify \eqref{tight6}, we claim that we can 
 find a constant $ C:=C(H)>0 $ such that 
\begin{equation}\label{tight7}
|N^{2} L_{N}\<\pi^{N}_{r},H\>| \leq C
\end{equation} 
and \eqref{tight6} easily follows.
To prove \eqref{tight7}, we begin by handling 
\begin{equation}\label{gerador_interno_eq}
 N^2 L_{N,0} \< \pi^{N}_{s}, H \> 
=\pfrac{1}{N}\sum_{x=2}^{N\!-\!2}\eta_s(x)\Delta_N H\left(\pfrac{x}{N}\right)-\eta_{s}(N\!-\!1)\nabla_N^- H(\pfrac{N\!-\!1}{N})+\eta_{s}(1)\nabla_N^+ H(\pfrac{1}{N}),
\end{equation}
where
\begin{equation*}
\Delta_N H(u):=N^2\left[H\left(u+\pfrac{1}{N}\right) + H\left(u-\pfrac{1}{N}\right)-2H(u)\right],
\end{equation*}
\begin{equation*}
\nabla_N^+H(\pfrac{1}{N}):=N\left( H\left( \pfrac{2}{N} \right) - H(\pfrac{1}{N}) \right)
\end{equation*}
and
\begin{equation*}
\nabla_N^-H(\pfrac{N\!-\!1}{N}):=-N\left( H\left( \pfrac{N\!-\!2}{N} \right) - H(\pfrac{N\!-\!1}{N}) \right).
\end{equation*} 

  The mean value theorem implies that \eqref{gerador_interno_eq} is bounded by $2\Vert H''\Vert_{\infty} + 2\Vert H' \Vert_{\infty}$.
To finish the proof of \eqref{tight7}, we need only to verify that $ N^{2} L_{N,b} \< \pi^{N}_{s}, H \> $ is bounded by $ 2c\Vert H \Vert_{\infty} $.
 This is a consequence of 
\begin{equation}\label{gerador_fronteira_eq}
\begin{split}
 N^{2} L_{N,b} \< \pi^{N}_{s}, H \> = cN^{1-\theta} (\alpha-\eta_{s}(1))H(\pfrac{1}{N})+
cN^{1-\theta}(\beta-\eta_{s}(N\!-\!1))H(\pfrac{N\!-\!1}{N}).
\end{split}
\end{equation}
\end{proof}

\begin{proof}[Tightness for $ \theta \in[0,1)$]
	If we try to apply the same strategy used for $ \theta \geq 1 $, we will run into trouble when trying to control the modulus of continuity of $ \int_0^t N^2 L_{N,b}\< \pi^N_s, H\>\,ds $, because the expression in \eqref{gerador_fronteira_eq} can explode with $ N $. However, that expression vanishes if the test function $ H $ has compact support in $ (0,1) $. Therefore, we can reuse the computations for $ \theta \geq 1 $ to show that \eqref{tight2} is true when $ H \in C^2_c(0,1) $. To extend the validity of \eqref{tight2} to $ H \in C[0,1] $ we can use the estimate
	\begin{equation*}
	|\<\pi^N_t, H\> - \< \pi^N_s, H\>| \leq \frac{1}{N}\sum_{x=1}^{N\!-\!1}|H(\pfrac{x}{N})|,
	\end{equation*}
and approximate $ H $ in $ L^1[0,1] $ (with the Lebesgue measure) by functions in $ C^2_c(0,1) $.

\end{proof}

\begin{nota}
{\it{ Markov's inequality and the fact that $\lim_{N\to\infty} \EE_{\mu_{N}} \left[ \sup_{0 \leq t \leq T} |M^{N}_{t}(H)|^2\right]^{\sfrac{1}{2}}=0 $ imply that
\begin{equation}\label{limprob}
\lim_{N\rightarrow\infty}\PP_{\mu_N}\left[\sup_{0\leq t\leq T}
|M^{N}_{t}(H)|>\delta\right]=0,\quad\mbox{for all }\delta >0.
\end{equation}}}
\end{nota}

\section{Replacement Lemma and Sobolev Space}\label{s5}

In this section we establish the Replacement Lemma (see Lemma \ref{replace1} for the regime $ \theta \in[0,1) $ and Lemma $ \ref{substituicao} $ for $ \theta \geq 1 $) and prove that the limiting densities of the empirical measure are in the Sobolev Space $L^2(0,T;\mathcal{H}^1)$ (Proposition \ref{prop6.8}).
Without loss of generality, we suppose that $\alpha\leq \beta$.  Here we need to use that $0<\alpha$ and $\beta<1$.  Let $\nu_{\gamma(\cdot)}^N$ 	be the Bernoulli product measure on $\Omega_N$ with marginals given by	
 $$\nu_{\gamma(\cdot)}^N\{\eta(x)=1\}=\gamma(\pfrac{x}{N}).$$ 
Denote by $H_N (\mu_N | \nu_{\gamma(\cdot)}^N)$ the relative entropy of a probability
measure $\mu_N$ on $\Omega_N$ with respect to the probability measure $\nu_{\gamma(\cdot)}^N$. For the definition and properties of the relative entropy, we
refer the reader to  Appendix 1 in \cite{kl}. We can obtain a standard bound for $H_N (\mu_N | \nu_{\gamma(\cdot)}^N)$, which says that there
exists a finite constant $K_0:=K_0(\alpha,\beta)$, such that
\begin{equation}
\label{f06}
H_N (\mu_N | \nu_{\gamma(\cdot)}^N) \;\le\; K_0 N,
\end{equation}
for any probability measure $\mu_N$ on ${\Omega_N}$.

\subsection{Dirichlet form}
The Dirichlet form $\< L_{N}g,g\>_{\mu}$ does not always have a closed form. In this subsection we compare the Dirichlet form with the closed forms defined as
\begin{eqnarray*}
\!\!\!\!\!\!\!\!\!\!\!\!\!\! &&
D_{N,0}(g,\mu)\;:=\; \frac{ 1}{2}\sum_{x=1}^{N\!-\!2} 
\sum_{\eta\in\Omega_N} \Big( g(\eta^{x,x+1})-
g(\eta)\Big)^2 \, \mu (\eta),
\end{eqnarray*}
\begin{eqnarray*}
\!\!\!\!\!\!\!\!\!\!\!\!\!\! &&
D_{N,b}^\alpha(g,\mu)\;:=\; \frac{ 1}{2}
\sum_{\eta\in\Omega_N}  cN^{-\theta} r_\alpha(\eta)\Big( g(\eta^{1}) -
g(\eta)\Big)^2 \, \mu (\eta),
\end{eqnarray*}
and
\begin{eqnarray*}
\!\!\!\!\!\!\!\!\!\!\!\!\!\! &&
D_{N,b}^\beta(g,\mu)\;:=\; \frac{ 1}{2}
\sum_{\eta\in\Omega_N} cN^{-\theta} r_\beta(\eta)\Big(g(\eta^{N\!-\!1}) -
g(\eta) \Big)^2 \, \mu(\eta),
\end{eqnarray*}
where   $\mu$ is any measure on $\Omega_N$  and $g:\Omega_N\to\mathbb{R}$ is any function. 

\begin{lema}\label{lema41}
Let $\gamma:[0,1]\to(0,1)$ be a function.
Let $f:\Omega_N\to\RR_+$ be a density with respect to the measure $\nu_{\gamma(\cdot)}^N$. Then
\begin{enumerate}
\item[i)]  If $\gamma$  is a constant function,
then
\begin{equation*}
 \big\< L_{N,0} \sqrt f ,\, \sqrt f\,\big\>_{\nu_{\gamma(\cdot)}^N}\;=\; - 
 D_{N,0}(\sqrt{f},\nu_{\gamma(\cdot)}^N);
\end{equation*}
\item[ii)] if $\gamma$ is smooth, then  there exists a constant $C_0>0$ (which depends on $\gamma$) such that 
\begin{equation*}
 \big\< L_{N,0} \sqrt f ,\, \sqrt f\,\big\>_{\nu_{\gamma(\cdot)}^N}\;=\; - \frac{1}{2}
 D_{N,0}(\sqrt{f},\nu_{\gamma(\cdot)}^N)+\mc E_N(\gamma),
\end{equation*}
with
$
|\mc E_N(\gamma)|\leq\frac{C_0}{N}.
$
\end{enumerate}

\end{lema}
 \begin{proof}
 	
Writing $ \< L_{N,0}\sqrt{f}, \sqrt{f}\>_{\nu_{\gamma(\cdot)}^N} = \frac{1}{2}\< L_{N,0}\sqrt{f}, \sqrt{f}\>_{\nu_{\gamma(\cdot)}^N} + \frac{1}{2}\< L_{N,0}\sqrt{f}, \sqrt{f}\>_{\nu_{\gamma(\cdot)}^N} $, making the change of variables $ \eta \mapsto \eta^{x,x+1} $ in the second term, adding and subtracting the term 
\begin{equation*} \frac{1}{2}
\sum_{x=1}^{N\!-\!2} \sum_{\eta \in \Omega_N}
\lrp{\sqrt{f(\eta^{x,x+1})} - \sqrt{f(\eta)}
	} \sqrt{f(\eta^{x,x+1})}\;
\nu_{\gamma(\cdot)}^N(\eta)
\end{equation*} 
and using the definition of $D_{N,0}$, then  $\< L_{N,0}\sqrt{f}, \sqrt{f}\>_{\nu_{\gamma(\cdot)}^N} $ can be rewritten as $-D_{N,0}(\sqrt{f},\nu_{\gamma(\cdot)}^N)+ g_N(\gamma)$, where
\begin{equation}\label{dirichlet_interno1}
g_N(\gamma)=\frac{1}{2}
\sum_{x=1}^{N\!-\!2} \sum_{\eta \in \Omega_N}
\lrp{\sqrt{f(\eta^{x,x+1})} - \sqrt{f(\eta)}
	} \sqrt{f(\eta^{x,x+1})}
\lrp{1 -\frac{ \nu_{\gamma(\cdot)}^N(\eta^{x,x+1})}{\nu_{\gamma(\cdot)}^N(\eta)}
	}\nu_{\gamma(\cdot)}^N(\eta).
\end{equation} 
 To handle $g_N(\gamma)$, 
we start observing that
\begin{equation*}
\Big|1-
 \frac{ \nu_{\gamma(\cdot)}^N(\eta^{x,x+1})}{\nu_{\gamma(\cdot)}^N(\eta)}\Big|\leq \tilde c_\gamma\,\Big\vert\gamma(\frac{x+1}{N})-\gamma(\frac{x}{N})\Big\vert,
\end{equation*} 
where $\tilde c_\gamma>0$ depends on $\gamma$.
Thus, if $\gamma$ is constant, then $ g_N(\gamma)=0$.
But,  if $\gamma$ is smooth, we  need to analyze  $ g_N(\gamma)$. Applying the Young's inequality:
\begin{equation}\label{young}
 ab \leq \pfrac{B}{2}a^2 +\pfrac{1}{2B} b^2,
 \end{equation} in  \eqref{dirichlet_interno1}, with $B=1$, $a=|\sqrt{f(\eta^{x,x+1})} - \sqrt{f(\eta)}|$ and  $b=\sqrt{f(\eta^{x,x+1})}\,\big|1 -\sfrac{ \nu_{\gamma(\cdot)}^N(\eta^{x,x+1})}{\nu_{\gamma(\cdot)}^N(\eta)}
	\big|$,   we get
$
 |g_N(\gamma)|\leq
\frac{1}{2} D_{N,0}(\sqrt f,\nu_{\gamma(\cdot)}^N) +
\mc E_N(\gamma),
$ where
\begin{equation*}
 \mc E_N(\gamma)=
\frac{1}{2}
\sum_{x=1}^{N\!-\!2} \sum_{\eta \in \Omega_N}
f(\eta) \,\Big|1-
 \frac{ \nu_{\gamma(\cdot)}^N(\eta)}{\nu_{\gamma(\cdot)}^N(\eta^{x,x+1})}\Big|^2\,\frac{\nu_{\gamma(\cdot)}^N(\eta^{x,x+1})}{\nu_{\gamma(\cdot)}^N(\eta)}\,
\nu_{\gamma(\cdot)}^N(\eta).
\end{equation*}
In the preceding line we performed the change of variables $ \eta \mapsto \eta^{x,x+1} $.
To finish the proof we need to bound $\mc E_N(\gamma)$ by $C_0/N$. In order to do this, we use
\begin{equation}\label{bound}
\Big|1-
 \frac{ \nu_{\gamma(\cdot)}^N(\eta^{x,x+1})}{\nu_{\gamma(\cdot)}^N(\eta)}\Big|\leq \tilde c_\gamma\,\Vert\gamma'\Vert_{\infty}\frac{1}{N},
\end{equation}
 $|\sfrac{\nu_{\gamma(\cdot)}^N(\eta^{x,x+1})}{\nu_{\gamma(\cdot)}^N(\eta)}|\leq C_\gamma$ and
that  $ f $ is a density with respect to $ \nu^N_{\gamma(\cdot)}  $. 

 \end{proof}
 
  \begin{lema}\label{dirichlet_gerador_1}
  Let $\gamma:[0,1]\to(0,1)$ be a function.
For any  function $g\in L^2(\nu^N_{\gamma(\cdot)})$, we have
$$\< L_{N,b}^\alpha\,g,\, g \>_{\nu^N_{\gamma(\cdot)}} = -D_{N,b}^\alpha(g,\nu^N_{\gamma(\cdot)})+\mc E_N^\alpha(\gamma,g),$$
with $|\mc E_N^\alpha(\gamma,g)|\leq \,c\,C(\gamma)\,|\gamma(\pfrac{1}{N})-\alpha|\,N^{-\theta}\,\Vert g\Vert^2_{\nu_{\gamma(\cdot)}^N}$.
 A similar equality holds for $L_{N,b}^\beta$, just replacing $|\gamma(\pfrac{1}{N})-\alpha|$ by
$|\gamma(\pfrac{N-1}{N})-\beta|$.

\end{lema}

\begin{proof}
Define the sets
 $A_i=\{\eta\in\Omega_n;\, \eta(1)=i\}$, for $i=0,1$, and write $\< L_{N,b}^\alpha\,g,\, g \>_{\nu^N_{\gamma(\cdot)}}$ as
\begin{equation*}\begin{split}
&\sum_{\eta\in A_0}c\frac{\alpha}{N^\theta}\Big[g(\eta^1)g(\eta)-\pfrac{1}{2}g(\eta)^2\Big]\nu_{\gamma(\cdot)}^N(\eta)+\sum_{\eta\in A_1}c\frac{\alpha}{N^\theta}\Big[-\pfrac{1}{2}g(\eta^1)^2\Big]\frac{1-\gamma(\frac{1}{N})}{\gamma(\frac{1}{N})}\;\nu_{\gamma(\cdot)}^N(\eta)\\
+&\sum_{\eta\in A_1}c\frac{(1-\alpha)}{N^\theta}\Big[g(\eta^1)g(\eta)-\pfrac{1}{2}g(\eta)^2\Big]\nu_{\gamma(\cdot)}^N(\eta)+\sum_{\eta\in A_0}c\frac{(1-\alpha)}{N^\theta}\Big[-\pfrac{1}{2}g(\eta^1)^2\Big]\frac{\gamma(\frac{1}{N})}{1-\gamma(\frac{1}{N})}\;\nu_{\gamma(\cdot)}^N(\eta).\\
\end{split}
\end{equation*}
If $\gamma(\frac{1}{N})=\alpha$, then $\< L_{N,b}^\alpha\,g,\, g \>_{\nu^N_{\gamma(\cdot)}}=-D_{N,b}^\alpha(g,\nu^N_{\gamma(\cdot)})$.
For a general $\gamma$, we have that  $\< L_{N,b}^\alpha\,g,\, g \>_{\nu^N_{\gamma(\cdot)}}$ is equal to $-D_{N,b}^\alpha(g,\nu^N_{\gamma(\cdot)})$ plus
\begin{equation*}\begin{split}
&\sum_{\eta\in A_1}c\frac{\alpha}{N^\theta}\Big[-\pfrac{1}{2}g(\eta^1)^2\Big]\Big[\frac{1-\gamma(\frac{1}{N})}{\gamma(\frac{1}{N})}-\frac{1-\alpha}{\alpha}\Big]\,\nu_{\gamma(\cdot)}^N(\eta)\\
+&\sum_{\eta\in A_0}c\frac{(1-\alpha)}{N^\theta}\Big[-\pfrac{1}{2}g(\eta^1)^2\Big]\Big[\frac{\gamma(\frac{1}{N})}{1-\gamma(\frac{1}{N})}-\frac{\alpha}{1-\alpha}\Big]\,\nu_{\gamma(\cdot)}^N(\eta).\\
\end{split}
\end{equation*}
The expression above is bounded by
\begin{equation*}\begin{split}
&c\,|\gamma(\pfrac{1}{N})-\alpha|\,\frac{C(\gamma)}{N^\theta}\sum_{\eta\in \Omega_N}g(\eta)^2\,\nu_{\gamma(\cdot)}^N(\eta).
\end{split}
\end{equation*}
Above we have $g(\eta)$ instead of $g(\eta^1)$, because  we performed a change of variables and estimate the error, as we did in Lemma \ref{lema41}.

\end{proof}

\begin{cor}\label{dirichlet_gerador}
Let  $\nu^N_\beta$ the Bernoulli product measure in $\Omega_N$ with constant parameter $\beta\in(0,1)$. Then, there exists a positive constant $C_0=C(\alpha,\beta,c)$ such that
$$\big\< L_{N,b}^\alpha\,\sqrt f,\, \sqrt f \,\big\>_{\nu^N_{\beta}} = -D_{N,b}^\alpha(\sqrt f,\nu^N_{\beta})+O_{\alpha,\beta,c}(\pfrac{1}{N^\theta}),$$
for all  $f$ density with respect to $\nu^N_\beta$.
\end{cor}
\begin{proof}
Follows from Lemma \ref{dirichlet_gerador_1}, taking the function $\gamma$ constant equal to $\beta$. 
\end{proof}

\subsection{Replacement Lemma for $\theta\in[0,1)$}\label{replace111}
Let $\gamma:[0,1]\to [0,1]$ be a smooth function such that $\alpha\leq \gamma(u)\leq\beta$, for all $u\in[0,1]$, and assume that there exists a neighborhood of $0$ where the function $\gamma $ is equal to $\alpha$ and a neighborhood of $1$, where the function $\gamma $ is equal to $\beta$.

\begin{lema}[Replacement Lemma]\label{replace1}
Suppose $\theta\in[0,1)$. For all function $G\in C([0, T] )$, $t\in[0,T]$ and $x\in\{1,N-1\}$,  we have
\begin{equation*}
 \varlimsup_{N\to\infty}
\mathbb E_{\mu_N}\Big[\,\Big|\int_0^t G_s\;\{\eta_s(x)-\zeta_x\}\, ds\Big|\,\Big]\;=\;0,
\end{equation*}
where $\zeta_1=\alpha$ and $\zeta_{N\!-\!1}=\beta$.
\end{lema}

In the proof of this lemma we use the following estimates:
\begin{lema}\label{corcor}Let $f$ a density function with respect to $\nu_{\gamma(\cdot)}^N$.
For any $B>0$, there are non-negative constants $C_0=C(\gamma),C_1=C_1(\alpha),C_2=C_2(\alpha),C_3=C_3(\beta),C_4=C_4(\beta)$  such that
  \begin{equation*}
\begin{split}
\Big|\int
&(\eta(1)-\alpha)\,f(\eta)\;d \nu_{\gamma(\cdot)}^N(\eta)\Big| \leq B\, C_1-\frac{C_2}{B}\,N^\theta\,\big\< L_{N}\,\sqrt f,\, \sqrt f \,\big\>_{\nu^N_{\gamma(\cdot)}}+\frac{C_0C_2}{B}\,N^{\theta-1}\\
\end{split}
\end{equation*}
and
  \begin{equation*}
\begin{split}
\Big|\int
&(\eta(N-1)-\beta)\,f(\eta)\;d \nu_{\gamma(\cdot)}^N(\eta)\Big|\leq B\, C_3-\frac{C_4}{B}\,N^\theta\,\big\< L_{N}\,\sqrt f,\, \sqrt f \,\big\>_{\nu^N_{\gamma(\cdot)}}+\frac{C_0C_4}{B}\,N^{\theta-1}.
\end{split}
\end{equation*}
 \end{lema}
 
The proof of these estimates we postponed to the end of this subsection. Now, we present the proof of the Replacement Lemma (Lemma \ref{replace1}).

\begin{proof}[Proof of Lemma \ref{replace1}]
We will only prove the  assertion for $x=1$, the other one is similar.
From Jensen's inequality and the definition of entropy, for any $A\in\mathbb R$ (which will be chosen large),
the expectation considered in the statement of the lemma is bounded from above by
\begin{equation}\label{log}
\frac{H_N(\mu_N|\nu_{\gamma(\cdot)}^N)}{A N}+\frac{1}{A N}\log \mathbb E_{\nu_{\gamma(\cdot)}^N} \Big[
\exp\Big\{A\,N\Big|\int_0^t G_s\;\{
\eta_s(1)-\alpha\}\,ds\Big|\Big\}\Big].
\end{equation}
By \eqref{f06}, the first term above is bounded by $ K_0/A$, so we only need to show that the second term vanishes as $ N\to\infty $.
Since $e^{|x|}\leq e^x+e^{-x}$ and
\begin{equation}\label{log bounds}
\varlimsup_N \pfrac{1}{N}\log (a_N+b_N)= \max\Big\{
\varlimsup_N \pfrac{1}{N}\log a_N,\varlimsup_N \pfrac{1}{N}\log b_N\Big\},
\end{equation}
for all sequences $\{a_N\}$ and $\{b_N\}$ of positive numbers,
we can study the second term in \eqref{log} without the absolute value inside the exponential. 
Using Feynman-Kac's Lemma (see Lemma \ref{FK-sem med inv}),
we reduce the problem to that of estimating
\begin{equation*}
T\;\sup_{f\textrm{~density}} \,\Bigg\{\Vert G\Vert_\infty \,\Big|\int (\eta_1 - \alpha)\,f(\eta)\,d\nu_{\gamma(\cdot)}^N(\eta)\Big|+\frac{N}{A}\, \big\<L_N\sqrt{f},\sqrt{f}\,\big\>_{\nu_{\gamma(\cdot)}^N}\Bigg\}.
\end{equation*}
Using Lemma \ref{corcor}, we can bound the last expression by
\begin{equation*}
\begin{split}
T\,\sup_{f\textrm{~density}} \,\Bigg\{\Vert G\Vert_\infty \Big(B\, C_1-\frac{2C_2}{B}\,N^\theta\,\big\< L_{N}\,\sqrt f,\, \sqrt f \,\big\>_{\nu^N_{\gamma(\cdot)}}+\frac{2C_0C_2}{B}\,N^{\theta-1}\Big)+\frac{N}{A}\, \big\<L_N\sqrt{f},\sqrt{f}\,\big\>_{\nu_{\gamma(\cdot)}^N}\Bigg\}.
\end{split}
\end{equation*}
Taking $B=2C_2\Vert G\Vert_\infty A\,N^{\theta-1}$, the last expression becomes
\begin{equation*}
\begin{split}
\frac{TC_0}{A} +\, 2TC_1C_2\Vert G\Vert_\infty^2 \,\frac{A}{N^{1-\theta}}.
\end{split}
\end{equation*}
The statement follows if we choose $A=N^{\sfrac{(1-\theta)}{2}}$.
\end{proof}

The inequalities in
 Lemma \ref{corcor} are similar, then we will focus only in the first one, which deals with the left-hand side of the boundary.

\begin{lema}\label{lemausual} There exist  a constant $C_1=C_1(\alpha)>0$ such that, for all   density functions $f$  with respect to $\nu_{\gamma(\cdot)}^N$ and  for all $B>0$, we have
\begin{equation*}
\begin{split}
\Big|\int
&(\eta(1)-\alpha)\,f(\eta)\;d \nu_{\gamma(\cdot)}^N(\eta)\Big|\leq B\, C_1+\frac{1}{2B}\var(\sqrt{f_1}),\\
\end{split}
\end{equation*}
where $f_1$ is the expectation of the function $f$ conditioned to the occupation in the site $1$:
\begin{equation*}
\begin{split}
f_1(\eta(1)):=E_{\nu_{\gamma(\cdot)}^N}[f|\eta(1)]&=
\sum_{\bar{\eta}}f(\eta(1),\bar{\eta})\;\overline{\nu_{\gamma(\cdot)}^N}(\bar{\eta}).
\end{split}
\end{equation*}
In the expression above: $\eta=(\eta(1),\bar{\eta})\in \{0,1\}\times\{0,1\}^{\{2,\dots, N\!-\!1\}}$ and the measure
$\nu_{\gamma(\cdot)}^N(\eta)=\nu_\alpha^1(\eta(1))\overline{\nu_{\gamma(\cdot)}^N}(\bar{\eta})$.
\end{lema}
To obtain  the Lemma \ref{corcor} from the last lemma, we need to relate $ \var(\sqrt{f_1})$ with the Dirichlet form. To deal with this question, we will use the next lemma together with Lemmas \ref{lema41} and \ref{dirichlet_gerador_1}. 

\begin{lema}\label{lema42}
 Let $f$ be a density with respect to the  measure $\nu_{\gamma(\cdot)}^N$. Then there exists a constant $C_2=C_2(\alpha)>0$ such that
\begin{equation*}
\var(\sqrt{f_1})\;\leq\; C_2\,N^\theta\,D_{N,b}^\alpha(\sqrt{f},\nu_{\gamma(\cdot)}^N).
\end{equation*}
\end{lema}
Finally, we are in position to prove Lemma \ref{corcor}.

 \begin{proof}[Proof of Lemma \ref{corcor}]
 From Lemmas \ref{lemausual} and \ref{lema42}, we have
  \begin{equation*}
\begin{split}
\Big|\int
&(\eta(1)-\alpha)\,f(\eta)\;d \nu_{\gamma(\cdot)}^N(\eta)\Big|\leq B\, C_1+\frac{1}{2B}C_2\,N^\theta\,D_{N,b}^\alpha(\sqrt{f},\nu_{\gamma(\cdot)}^N)\leq B\, C_1+\frac{1}{2B}C_2\,N^\theta\,D_{N}(\sqrt{f},\nu_{\gamma(\cdot)}^N),\\
\end{split}
\end{equation*}
where $D_{N}:= D_{N,b}^\alpha+D_{N,b}^\beta+D_{N,0}$. The last inequality is true because $D_{N,b}^\beta(\sqrt{f},\nu_{\gamma(\cdot)}^N)$ and $D_{N,0}(\sqrt{f},\nu_{\gamma(\cdot)}^N)$ are non-negative.
By the choice of $\gamma$ and  Lemmas \ref{lema41} and \ref{dirichlet_gerador_1}, we have that $D_{N}(\sqrt{f},\nu_{\gamma(\cdot)}^N)$ is equal to $-\< L_{N,b}^\alpha\,\sqrt f,\, \sqrt f \>_{\nu^N_{\gamma(\cdot)}} -\< L_{N,b}^\beta\,\sqrt f,\, \sqrt f \>_{\nu^N_{\gamma(\cdot)}}-2\< L_{N,0}\,\sqrt f,\, \sqrt f \>_{\nu^N_{\gamma(\cdot)}} -\,\mc E_N(\gamma)$. Then 
 \begin{equation*}
\begin{split}
\Big|\int
&(\eta(1)-\alpha)\,f(\eta)\;d \nu_{\gamma(\cdot)}^N(\eta)\Big|\leq B\, C_1-\frac{2}{2B}C_2\,N^\theta\,\big\< L_{N}\,\sqrt f,\, \sqrt f \,\big\>_{\nu^N_{\gamma(\cdot)}}+\frac{2}{2B}C_2\,N^\theta\,\frac{C_0}{N} ,\\
\end{split}
\end{equation*}
because $|\mc E_N(\gamma)|\leq \frac{C_0}{N}  $ and $ -\< L_{N,b}^\alpha\,\sqrt f,\, \sqrt f \>_{\nu^N_{\gamma(\cdot)}} -\< L_{N,b}^\beta\,\sqrt f,\, \sqrt f \>_{\nu^N_{\gamma(\cdot)}}\leq 0 $.

 \end{proof}

To the end of this subsection we will present the proofs of the Lemmas \ref{lemausual} and \ref{lema42}.

\begin{proof}[Proof of Lemma \ref{lemausual}]
We will prove  the inequality for $\eta(1)-\alpha$ without absolute value, because we can obtain the same bound  for $\alpha-\eta(1)$. 
From definition of $f_1$ in the statement of Lemma \ref{lemausual}, we have 
\begin{equation*}
\begin{split}
\int
\{\eta(1)-\alpha\}f(\eta)d \nu_{\gamma(\cdot)}^N(\eta) &=
\sum_{\eta(1)}\{\eta(1)-\alpha\} \;\sum_{\bar{\eta}}f(\eta(1),\bar{\eta})\;\overline{\nu_{\gamma(\cdot)}^N}(\bar{\eta})\;\;\nu_\alpha^1(\eta(1))\\
&=
\sum_{\eta(1)}\{\eta(1)-\alpha\} f_1(\eta(1))\nu_\alpha^1(\eta(1)).\\
\end{split}
\end{equation*}
Since the function $\{\eta(1)-\alpha\}$ has zero mean, the last expression becomes
\begin{equation*}
\begin{split}
\sum_{\eta(1)}\{\eta(1)-\alpha\}\big\{ f_1(\eta(1))-E_{\nu_\alpha^1}[\sqrt{f_1}]^2\big\}\nu_\alpha^1(\eta(1)).\\
\end{split}
\end{equation*}
Let $B>0$. Using the Young's inequality \eqref{young}, with $a=\{\eta(1)-\alpha\}\{\sqrt{ f_1(\eta(1))}+E_{\nu_\alpha^1}[\sqrt{f_1}]\}$ and $b=\sqrt{ f_1(\eta(1))}-E_{\nu_\alpha^1}[\sqrt{f_1}]$, the sum above is bounded from above by
\begin{equation*}
\begin{split}
\frac{B}{2}\sum_{\eta(1)}\{\eta(1)-\alpha\}^2\big\{\sqrt{ f_1(\eta(1))}+E_{\nu_\alpha^1}[\sqrt{f_1}]\big\}^2\nu_\alpha^1(\eta(1))+\frac{1}{2B}\var(\sqrt{f_1}).\\
\end{split}
\end{equation*}
Now, as $f$ is a density with respect to $\nu_{\gamma(\cdot)}^N$, we have
\begin{equation*}
\begin{split}
& \frac{1}{2}\sum_{\eta(1)}\{\eta(1)-\alpha\}^2\{\sqrt{ f_1(\eta(1))}+E_{\nu_\alpha^1}[\sqrt{f_1}]\big\}^2\nu_\alpha^1(\eta(1)) \leq C_1(\alpha),
\end{split}
\end{equation*}
which finishes the proof.
\end{proof}

 \begin{proof}[Proof of Lemma \ref{lema42}]
 In order to prove the inequality in the statement of this lemma,
we compute the variance for $\sqrt{f_1}$:
\begin{equation}\label{var1}
\begin{split} 
&\var(\sqrt{f_1})=
\sum_{\eta(1)} \Big(\sqrt{f_1(\eta(1))}-E_{\nu_\alpha^1}[\sqrt{f_1}]\Big)^2\;\nu_\alpha^1(\eta(1))\\
&=
(1-\alpha)^2\alpha \Big(\sqrt{f_1(1)}-\sqrt{f_1(0)}\Big)^2+
(1-\alpha)\alpha^2 \Big(\sqrt{f_1(0)}-\sqrt{f_1(1)}\Big)^2\\
&\leq c(\alpha)\, N^\theta\sum_{\eta(1)}c\Big[\pfrac{\alpha}{N^\theta}(1-\eta(1))+\pfrac{1-\alpha}{N^\theta}\eta(1)\Big]\Big( \sqrt{f_1(1-\eta(1))} -
\sqrt{f_1(\eta(1))} \Big)^2 \,\nu_\alpha^1(\eta(1)).
\end{split}
\end{equation}
 Keep in mind the definition of $f_1$ and observe that
\begin{equation}\label{41aas}
\begin{split}
&\Big( \sqrt{ f_1(1-\eta(1))} -
\sqrt{f_1(\eta(1))} \Big)^2 \\&=\Bigg( \Big(\sum_{\bar{\eta}}f(1-\eta(1),\bar{\eta})\;\overline{\nu_{\gamma(\cdot)}^N}(\bar{\eta})\Big)^{1/2} -
\Big(\sum_{\bar{\eta}}f(\eta(1),\bar{\eta})\;\overline{\nu_{\gamma(\cdot)}^N}(\bar{\eta})\Big)^{1/2} \Bigg)^2 \\
&\leq \sum_{\bar{\eta}}\Bigg( \Big(f(1-\eta(1),\bar{\eta})\Big)^{1/2} -
\Big(f(\eta(1),\bar{\eta})\Big)^{1/2} \Bigg)^2\;\overline{\nu_{\gamma(\cdot)}^N}(\bar{\eta}). \\
\end{split}
\end{equation}
The last inequality follows from 
$$\Big| \,\Vert g_1\Vert_{L^2(\mu)}-\Vert g_2\Vert_{L^2(\mu)}\Big|\,\leq\,\Vert g_1-g_2\Vert_{L^2(\mu)},$$
choosing the suitable functions $g_1$ and $g_2$ and measure $\mu$.
Thus, using the inequality \eqref{41aas} in the expression \eqref{var1}, we bound $\var(\sqrt{f_1})$ by
\begin{equation*}
\begin{split}
 c(\alpha)\, N^\theta\sum_{\eta}c\Big[\pfrac{\alpha}{N^\theta}(1-\eta(1))+\pfrac{1-\alpha}{N^\theta}\eta(1)\Big]\Big( \sqrt{f(\eta^1)} -\sqrt{f(\eta)} \Big)^2 \nu_{\gamma(\cdot)}^N(\eta).
\end{split}
\end{equation*}
Then, the statement of this lemma follows from last expression combined with the expression for $D_{N,b}^\alpha(\sqrt{f},\nu_{\gamma(\cdot)}^N)$ and for $r_\alpha(\eta)$.
 \end{proof}

\subsection{Replacement Lemma for $\theta \in[1,\infty)$}

For the following we  define the empirical density at the boundary $\{1,N\!-\!1\}$ as
\begin{equation*}
 \eta^{\varepsilon N}(1)\;=\;\frac{1}{\varepsilon N}\sum_{y=1}^{\floor{\varepsilon N}}\eta(y)
\qquad
\mbox{and}
\qquad
 \eta^{\varepsilon N}(N\!-\!1)\;=\;\frac{1}{\varepsilon N}\sum_{y=N\!-\!1-\floor{\varepsilon N}}^{N\!-\!1 }\eta(y).
\end{equation*}

The next lemma allows us to replace the random variables $\eta_s(1)$ and $\eta_s(N\!-\!1)$ by functions of the empirical measure, namely $\eta^{\varepsilon N}(1)$ and $\eta^{\varepsilon N}(N\!-\!1)$.

\begin{lema}[Replacement Lemma]\label{substituicao}
Suppose $\theta \geq 1$.  For all function $G\in C([0, T] )$, $t\in[0,T]$ and $x\in\{1,N-1\}$,  we have
\begin{equation*}
	\varlimsup_{\varepsilon \downarrow 0}\varlimsup_{N \to \infty} 
			\EE_{\mu_N}\left[
			\Big|\int_{0}^{t} G(s)[\eta_{s}^{\varepsilon N}(x)-\eta_{s}(x)]\,ds\Big|
		\right]=0.
\end{equation*}
\end{lema}

\begin{proof}

We will work in the case $x=1$, the other one is similar.
As in Lemma \ref{replace1}, we use the definition of relative entropy, Jensen's inequality and Feynman-Kac's Lemma, reducing the problem to that of estimating
\begin{equation*}
\frac{1}{A N}H(\mu_N | \nu_\beta^N)+
\int_0^t\;\sup_{f\textrm{~density}} \,\Big\{G(s) \int [\eta^{\varepsilon N}(1)-\eta(1)]\,f(\eta)\,d\nu_\beta^N(\eta)+\pfrac{N}{A}\, \<L_N\sqrt{f},\sqrt{f}\>_{\nu_\beta^N}\Big\}\;ds,
\end{equation*} 
for all $A>0$.
From now on, this proof follows the same steps as the proof of Lemma \ref{replace1}, using  Lemma \ref{llll} instead of  Lemma \ref{corcor}. 
Choosing $B=N/A$ an using the bound \eqref{f06}, the expression above is bounded by
\begin{equation*}
\frac{K_0}{A }+
\frac{T C_{\alpha,\beta}}{A\,N^{\theta-1}}+2 T \norm{G}_{\infty}^2 \varepsilon A.
\end{equation*} 
Taking the limit when $N\to\infty$ after when $\varepsilon\to 0$, remains just $\frac{K_0}{A }$, as $A$ is any non-negative number, the proof is concluded.
\end{proof}

\begin{lema}\label{llll} For all $B>0$ and $\varepsilon>0$ and for all density function $f$ with respect to $\nu_{\beta}^N$, we have
\begin{equation}\label{dirichlet_form_estimate}
G(s)  \int [\eta^{\varepsilon N}(1)-\eta(1)]f(\eta)
	\,d\nu^N_\beta(\eta) 
\leq -B\< L_{N} \sqrt f ,\, \sqrt f\>_{\nu_{\beta}^N}+B\,\frac{ C_{\alpha,\beta}}{N^{\theta}}+\frac{2\varepsilon N \norm{G}_{\infty}^2}{B}.
\end{equation}
\end{lema}
\begin{proof} This proof  follows the ideas of \cite{fgn1}, so we omit several details here. We can write $\eta^{\varepsilon N}(1)-\eta(1)$ as a telescopic sum.  Writing this sum as twice its half  and making the change of variables $ \eta \mapsto \eta^{x,x+1} $, we can write the integral in the left-hand side of \eqref{dirichlet_form_estimate} as
\begin{align*}
& \frac{G(s)}{2\varepsilon N}
	\sum_{x=1}^{\lfloor \varepsilon N \rfloor}
	\sum_{y=1}^{x-1}
	\int
	[\eta(y+1)-\eta(y)][f(\eta) - f(\eta^{y,y+1})]
	\,d\nu^N_\beta(\eta).
\end{align*}
 Writing $ f(\eta)-f(\eta^{y,y+1}) = [\sqrt{f(\eta)}-\sqrt{f(\eta^{y,y+1})}][\sqrt{f(\eta)}+\sqrt{f(\eta^{y,y+1})}] $, using  Young's inequality \eqref{young}, for all $B>0$, and the fact that there is at most one particle per site, we can bound the last expression by
\begin{align}\label{caracterizacao11}
 & \frac{1}{2\varepsilon N}
	\sum_{x=1}^{\lfloor \varepsilon N \rfloor}
	\sum_{y=1}^{x-1} 
	B\int	
	[\sqrt{f}(\eta)-\sqrt{f(\eta^{y,y+1})}]^2
	\,d\nu^N_\beta(\eta) \nonumber
\\
  & + \frac{G(s)^2}{2\varepsilon N}
	\sum_{x=1}^{\lfloor \varepsilon N \rfloor}
	\sum_{y=1}^{x-1}
	\frac{1}{B}\int
	[\sqrt{f(\eta)}+\sqrt{f(\eta^{y,y+1})}]^2
	\,d\nu^N_\beta(\eta).
\end{align}
Using that $f$ is a density for $\nu_\beta^N$, the second term above is bounded by
$
\frac{2\varepsilon N \norm{G}_{\infty}^2}{B}.
$
 Letting the sum in $y$ run from $1$ to $N\!-\!2$, the first term in \eqref{caracterizacao11} is bounded by $B\,D_{N,0}(\sqrt{f},\nu^N_\beta)$.  Item $i)$ of  Lemmas \ref{lema41} and \ref{dirichlet_gerador_1} and Corollary \ref{dirichlet_gerador} imply $D_{N,0}(\sqrt{f},\nu^N_\beta)= -\< L_{N,0} \sqrt f ,\, \sqrt f\>_{\nu_{\beta}^N}$,
 $0\leq D_{N,b}^\beta(\sqrt f ,\nu_{\beta}^N) =-\< L_{N,b}^\beta \sqrt f ,\, \sqrt f\>_{\nu_{\beta}^N}$ and $0\leq D_{N,b}^\alpha(\sqrt f ,\nu_{\beta}^N) \leq-\< L_{N,b}^\alpha \sqrt f ,\, \sqrt f\>_{\nu_{\beta}^N}+\frac{ C_{\alpha,\beta}}{N^{\theta}}$. Then,
$$D_{N,0}(\sqrt{f},\nu^N_\beta)\leq -\big\< L_{N,0} \sqrt f ,\, \sqrt f\,\big\>_{\nu_{\beta}^N}
-\big\< L_{N,b}^\beta \sqrt f ,\, \sqrt f\,\big\>_{\nu_{\beta}^N}-\big\< L_{N,b}^\alpha \sqrt f ,\, \sqrt f\,\big\>_{\nu_{\beta}^N}+\frac{ C_{\alpha,\beta}}{N^{\theta}}.$$

\end{proof}

\subsection{Sobolev space}

  This subsection is devoted to prove that the boundary terms on the hydrodynamic equation are well defined. To do this, we will prove that the hydrodynamic profile $\rho_t$ belongs to a Sobolev space, almost surely in $t\in[0,T]$. Our proof follows \cite{fgn1} and \cite{fgn2}.

\begin{prop}\label{prop6.8}
Let $ \QQ_{*} $ be a limit point of $ \{\QQ_{N}\} $. Then the measure $\QQ_*$ is concentrated on paths $\rho(t,u) du$ such that
$\rho\in L^{2}(0;T;\HH^{1}(0,1)) $.
\end{prop}

The proof follows from  the  next lemma and   the Riesz Representation Theorem, as  in the papers \cite{fgn1} and \cite{fgn2}.

\begin{lema} For all $\theta\geq 0$, there is a positive constant $ \kappa >0 $ such that
\label{s03}
\begin{equation*}
 E_{\mathbb Q^*} \Big[ \sup_H \Big\{ \int_0^T  \int_0^1 \partial_u H (s, u)  \, \rho_s(u)\,du\,ds\, 
- \,  \kappa \int_0^T  \int_0^1 H (s, u)^2 \, du \,ds\Big\} \Big] \, \le \, \infty ,
\end{equation*}
where the supremum is carried over all functions $H$ in $ C_{c}^{0,2}([0,T]\times(0,1))$.
\end{lema}
\begin{proof} 
Consider a sequence $\{H_\ell, \, \ell\ge 1\}$  dense in  $ C_{c}^{0,2}([0,T]\times(0,1))$ 
with respect to the norm $\Vert H\Vert_\infty+ \Vert \partial_u H \Vert_\infty$.  Thus, it is sufficient to prove that, for every $k\ge 1$,
\begin{equation}\label{to_prove}
 \begin{split}
  E_{\mathbb Q^*}\Big[ \max_{1\le i\le k} 
\Big\{ \int_0^T  \int_{[0,1]} \partial_u H_i (s, u)  \, \rho_s(u)\,du\,ds\, 
- \kappa \int_0^T  \int_{[0,1]} H_i (s, u)^2 \, du \,ds\Big\}  \Big]\, \le \, C,
 \end{split}
\end{equation} for some constant $C>0$, independent of $k$.
Since the function
$$
\pi_{.}\;\mapsto\;\max_{1\le i\le k}\Big\{\int_{0}^{T}\int_{[0,1]}\partial_u H_i(s,u)\,d\pi_{s}(u)\,ds- \kappa\<\!\<H_i,H_i\>\!\>\Big\}
$$
is continuous and bounded in the Skorohod topology and 
 $\QQ_{N}$ converges
to $\QQ^*$, the expression in the left hand side of \eqref{to_prove} is equal to
\begin{equation}\label{EEE}
 \begin{split}
 \lim_{N\to +\infty} \mathbb E_{\mu_N}\Big[ \max_{1\le i\le k} 
\Big\{  \int_0^T\<\partial_u H_i(s,\cdot) ,\,\pi^N_s\>\,ds
-\,  \kappa\<\!\<H_i,H_i\>\!\>
\Big\} \Big].
  \end{split}
\end{equation}
By the relative entropy definition, Jensen's inequality and
$\exp\Big\{
\max_{1\le j\le k} a_j \Big\} \leq \sum_{1\le j\le k}
\exp\{a_j\}
$, the expectation  in \eqref{EEE} is bounded from above by 
\begin{equation*}
 \begin{split}
\frac{H(\mu_{N}|\nu_{\gamma(\cdot)}^N)}{N}\,+\, \frac{1}{N}\log \sum_{1\le i\le k} \mathbb E_{\nu_{\gamma(\cdot)}^N}\Big[ 
\exp\Big\{N\int_0^T\<\partial _u H_i(s,\cdot) ,\,\pi^N_s\>\,ds
-\,  \kappa N\<\!\<H_i,H_i\>\!\>\Big\}\Big],
 \end{split}
\end{equation*}
where the profile $\gamma$ is the same used in  Subsection \ref{replace111}.
We can bound the first term in the sum above by $K_0$, recall \eqref{f06}.
Using \eqref{log bounds}, it is enough to show, for a fixed function $ H $, that
\begin{equation*}
 \begin{split}
 \varlimsup_{N\to +\infty}\frac{1}{N}\log \mathbb E_{\nu_{\gamma(\cdot)}^N}\Big[ 
\exp\Big\{\int_0^T N\<\partial_u H(s,\cdot) ,\,\pi^N_s\>\,ds
-\, \kappa N\<\!\<H,H\>\!\>\Big\}\Big]\leq \tilde c
,
 \end{split}
\end{equation*}
for some constant $\tilde c>0$ independent of $H$.
 Then the result follows from the next lemma and the definition of the empirical measure.
\end{proof}

\begin{lema} For all $\theta\geq 0$ and  any function  $ H\in C_{c}^{0,2}([0,T]\times(0,1))$, there exists $ \kappa >0 $ such that
\begin{equation*}
 \begin{split}
 \varlimsup_{N\to +\infty}\frac{1}{N}\log \mathbb E_{\nu_{\gamma(\cdot)}^N}\Big[ 
\exp\Big\{\int_0^T \sum_{x=1}^{N\!-\!1}\eta_s(x) \partial_uH(s,\pfrac xN)\,ds- \kappa N\<\!\<H,H\>\!\>
\Big\}\Big]<C_0T,
 \end{split}
\end{equation*}
where  tthe profile $\gamma$ is the same used in  Subsection \ref{replace111} and the constant $C_0=C_0(\gamma)>0$ comes from item $ii$ of Lemma \ref{lema41}. 
\end{lema}

\begin{proof}
We begin by observe that by the choice of $\gamma$ ($\gamma$ is constant equal to $\alpha$ and $\beta$ in a neighborhood of $0$ and $1$, respectively), then  using Lemma \ref{dirichlet_gerador_1} the Dirichlet form at the boundary is non-positive. And, by Lemma \ref{lema41}  the Dirichlet form in the bulk, $ \big\< L_{N,0} \sqrt f ,\, \sqrt f\,\big\>_{\nu_{\gamma(\cdot)}^N}$, is bounded by $- \frac{1}{2}D_{N,0}(\sqrt{f},{\nu_{\gamma(\cdot)}^N})\!+\!C_0/N $. Now, applying Feynman-Kac's Lemma and using these observations,  
we get that   the expression in the statement of this lemma is bounded by
\begin{equation}\label{energy01}
\int_{0}^{T} \sup_{f ~\mbox{density}} \bigg\{ \frac{1}{N}\int \sum_{x=1}^{N\!-\!1}\partial_{u}H_{s}\left(\pfrac{x}{N}\right) \eta(x)f(\eta) \,d\nu_{\gamma(\cdot)}^N(\eta)
- \kappa\<\<H,H\>\> - \frac{N}{2}D_{N,0}(\sqrt{f},{\nu_{\gamma(\cdot)}^N})\!+\!C_0 \bigg\}\;ds,
\end{equation}
here is used $H_{s}\left(\pfrac{x}{N}\right) $ instead of $H\left(s,\pfrac{x}{N}\right) $. 
Writing $\frac 1N \partial_uH_s(\frac xN)=H(\frac{x+1}{N})-H(\frac xN) + O_H(N^{-2})$ and summing by parts, we can write the integral as 
\begin{displaymath}
\sum_{x=1}^{N\!-\!2}\int H_{s}\left(\pfrac{x}{N}\right)(\eta(x)-\eta(x+1))f(\eta)\,d\nu_{\gamma(\cdot)}^N +O_H(\pfrac{1}{N}),
\end{displaymath}
(the compact support of $H$ takes care of the boundary term).  As usual, we rewrite this sum as twice its half and make the change of variables $ \eta \mapsto \eta^{x,x+1} $ (which is not  invariant for $\nu_{\gamma(\cdot)}^N$). Thus, the last sum is equal to
\begin{equation}\label{A}
\begin{split}
&\frac{1}{2}\sum_{x=1}^{N\!-\!2}\int H_{s}\left(\pfrac{x}{N}\right)(\eta(x)-\eta(x+1))(f(\eta)-f(\eta^{x,x+1}))\,d\nu_{\gamma(\cdot)}^N(\eta) \\
&-\frac{1}{2}\sum_{x=1}^{N\!-\!2}\int H_{s}\left(\pfrac{x}{N}\right)(\eta(x)-\eta(x+1))f(\eta^{x,x+1})\Big(\frac{\nu_{\gamma(\cdot)}^N(\eta^{x,x+1})}{\nu_{\gamma(\cdot)}^N(\eta)}-1\Big)\,d\nu_{\gamma(\cdot)}^N(\eta).
\end{split}
\end{equation}
 By Young's inequality \eqref{young} and the  fact that there is at most one particle per site,  the first term in \eqref{A} is bounded by
\begin{multline*}
\frac{B}{2}\sum_{x=1}^{N\!-\!2}\int \left( H_{s}\left(\pfrac{x}{N}\right)\right)^2(\sqrt{f(\eta)}+\sqrt{f(\eta^{x,x+1})})^{2}\,d\nu_{\gamma(\cdot)}^N(\eta) 
+ \frac{1}{2B}\sum_{x=1}^{N\!-\!2}\int (\sqrt{f(\eta)}-\sqrt{f(\eta^{x,x+1})})^{2}\,d\nu_{\gamma(\cdot)}^N(\eta),
\end{multline*}
 for any $ B >0 $.
Here, we have a small technical problem: $f(\eta^{x,x+1})$ is not necessarily a density with respect to $\nu_{\gamma(\cdot)}^N$. However, one can use \eqref{bound} to prove that $\int f(\eta^{x,x+1})\,d\nu_{\gamma(\cdot)}^N\leq 1 + C/N$ for some $C>0$ independent of $N$, so that $\int (\sqrt{f(\eta)}+\sqrt{f(\eta^{x,x+1})})^{2}\,d\nu_{\gamma(\cdot)}^N$ is bounded uniformly in $ N $. Here and from now on, $ C $ denotes a positive number that may change from line to line but does not depend on $ N $ nor on $ s $.

Choosing $B=1/2N$, the first term of \eqref{A} is bounded by 
\begin{equation}\label{energy02}
\frac{C}{N}\sum_{x=1}^{N\!-\!2}\left( H_{s}\left(\pfrac{x}{N}\right)\right)^2+\frac{N}{2}D_{N,0}(\sqrt{f},\nu^N_{\gamma(\cdot)}).
\end{equation}
To handle the second term in \eqref{A}, we use \eqref{bound}, Young's inequality \eqref{young} and  $\int f(\eta^{x,x+1})\,d\nu_{\gamma(\cdot)}^N\leq 1 + C/N$, getting the bound
\begin{equation}\label{energy03}
\frac{C}{N}\sum_{x=1}^{N\!-\!2}\left( H_{s}\left(\pfrac{x}{N}\right)\right)^2+ C.
\end{equation}
Putting everything together, we bound \eqref{energy01} by
\begin{equation*}
\int_{0}^{T}  \sup_{f~~\mbox{density}} \left\{
\frac{C}{N}\sum_{x=1}^{N\!-\!2}\left( H_{s}\left(\pfrac{x}{N}\right)\right)^2-\kappa\int_0^1\left( H_{s}\left(u\right)\right)^2\,du+O_H(\pfrac{1}{N})+C
\right\}\;ds.
\end{equation*}
 To finish the proof we just need to use that $ \frac{1}{N}\sum_{x=1}^{N\!-\!2}\left( H_{s}\left(\pfrac{x}{N}\right)\right)^2\rightarrow \int_{0}^{1}\left( H_{s}\left(u\right)\right)^2\,du $.

\end{proof}

\section{Characterization of the  Limit Points}\label{s6}

This section deals with the characterization of  limit points in the three ranges of  $\theta\geq 0$. We will focus in the case $\theta=1$, since this is the critical one. The other ones have similar proofs. We will also present the proof of the cases $\theta\in(0,1]$ and $\theta\in(1,\infty)$ pointing out the main differences between these cases and the proof of the case $\theta=1$.

\subsection{Characterization of the  limit points for $\theta=1$} \label{limit points critical beta}
Now we look at the limit points of the sequence $\{\QQ_{N}\}_{N\in\NN}$. We would like to stress that by Proposition \ref{prop6.8}, if $\QQ^*$ is a limit point of $\{\QQ_{N}\}_{N\in\NN}$, then 
the support of $ \QQ^{*} $ is contained in the set of trajectories $ \pi \in D_{\MM}[0,T] $ such that
 $\pi_{t}$ is a Lebesgue absolutely continuous measure with density $ \rho_{t} $ in $ \HH^{1}(0,1) $, for almost surely $t\in [0,T]$.
\begin{prop}\label{caracterizacao_master} If $\QQ^*$ is a limit point of $\{\QQ_{N}\}_{N\in\NN}$, then it is true that
\begin{multline}\label{master_equation}
\QQ^*
\Bigg[ \pi_.:\< \rho_{t}, H_{t}\> - \< \gamma , H_{0}\>  -\int_{0}^{t} \< \rho_{s} , (\partial_{s}+\Delta)  H_{s} \> \,ds 
-\int_{0}^{t}\big\{\rho_s(0)\partial_{u}H_{s}(0)
-\rho_{s}(1)\partial_{u}H_{s}(1)\big\}\,ds\\
-c\int_{0}^{t}\Big\{H_{s}(0)\big(\alpha-\rho_{s}(0)\big)+
H_{s}(1)\big(\beta-\rho_{s}(1)\big)\Big\}\,ds
=0, 
\;\;
\forall t\in[0,T], 
\;\;
\forall H\in C^{1,2}([0,T]\times [0,1])
\Bigg]=1.
\end{multline}
\end{prop}

\begin{proof}
 By density, all we need to verify is that, for $ \delta > 0 $ and $H\in C^{1,3}([0,T]\times [0,1])$ fixed, 
\begin{multline*}
\QQ^*
\Bigg[
	\pi_.:\sup_{0\leq t \leq T}
		\Big|
			\<\rho_t, H_t\> - \<\gamma, H_0\>
			-\int_0^t \< \rho_s, (\partial_s + \Delta)H_s \>\,ds
			-\int_0^t\rho_s(0)[\partial_u H_s(0)-cH_s(0)]\,ds
			\\
			+\int_0^t\rho_s(1)[\partial_u H_s(1)+cH_s(1)]\,ds 
			-c\int_0^t [\alpha H_s(0)+\beta H_s(1)]\,ds 
		\Big|> \delta
\Bigg]=0.
\end{multline*}
We would like to work with the probabilities
$\QQ_{N}$, using Portmanteau's Theorem. Unfortunately the set inside the above probability is not an open set in the Skorohod space. In order to avoid this problem,  we add and subtract the averages $\frac{1}{\varepsilon}\pi_s[0,\varepsilon]$ and $\frac{1}{\varepsilon}\pi_s[1-\varepsilon,1]$ for a fixed $\varepsilon >0$. Then we can bound from above the previous probability by the sum of
\begin{multline*}
\QQ^*
\Bigg[
	\pi_.:\sup_{0\leq t \leq T}
		\Big|
			\<\pi_t, H_t\> - \<\pi_0, H_0\>
			-\int_0^t \< \pi_s, (\partial_s + \Delta)H_s \>\,ds
			-\int_0^t\frac{1}{\varepsilon}\pi_s[0,\varepsilon][\partial_u H_s(0)-cH_s(0)]\,ds
			\\
			+\int_0^t\frac{1}{\varepsilon}\pi_s[1-\varepsilon,1][\partial_u H_s(1)+cH_s(1)]\,ds 
			-c\int_0^t [\alpha H_s(0)+\beta H_s(1)]\,ds
		\Big| > \delta/4
\Bigg]
\end{multline*}
with the probability of three sets, each one of them decreasing as $\varepsilon\to 0$, to sets of null probability, because $\frac{1}{\varepsilon}\pi_s[0,\varepsilon]$ and $\frac{1}{\varepsilon}\pi_s[1-\varepsilon,1]$  are suitable averages of  $\rho\in L^{2}(0,T;\HH^{1}(0,1))$ around the boundary points 0 and 1 and 
$\QQ^*$ is a limit point of $\{\QQ_{N}\}_{N\in\NN}$, each measure $\QQ_{N}$ has as initial measure $\mu_N$ and $\{\mu_N\}_{N\in\NN}$ is a sequence associated to the initial profile $\gamma:[0,1]\to\RR$. Now, we claim that we can use Portmanteau's Theorem  in order to conclude that the previous probability is bounded from above by
\begin{multline*}
\liminf_{N\to\infty}
\QQ_N
\Bigg[
	\pi_.:\sup_{0\leq t \leq T}
		\Big|
			\<\pi_t, H_t\> - \<\pi_0, H_0\>
			-\int_0^t \< \pi_s, (\partial_s + \Delta)H_s \>\,ds \\
			-\int_0^t\frac{1}{\varepsilon}\pi_s[0,\varepsilon][\partial_u H_s(0)-cH_s(0)]\,ds
			+\int_0^t\frac{1}{\varepsilon}\pi_s[1-\varepsilon,1][\partial_u H_s(1)+cH_s(1)]\,ds \\
			-c\int_0^t [\alpha H_s(0)+\beta H_s(1)]\,ds
		\Big| > \delta/16.
\Bigg].
\end{multline*}
Although the functions  $\frac{1}{\varepsilon}\ind{[0,\varepsilon]}$ and $\frac{1}{\varepsilon}\ind{[1-\varepsilon,1]}$ are not continuous, we can approximate them by continuous functions (linear by parts, say). The exclusion rule ensures that we are allowed to replace the indicator functions by their continuous approximations in the preceding probability. We can then apply Proposition A.3 of \cite{fgn1} and Portmanteau's Theorem. After this, we can change back to the indicator functions.
\footnote{Changing back to the indicator functions is a minor technical step. We only do that to work with $\eta^{\varepsilon N}_t$ instead of $\<\pi^N_t, \varphi\>$, where $\varphi$ denotes the continuous approximation to the indicator function. 
}

Summing and subtracting $ \int_0^t N^2 L_N \< \pi^N_s, H_s \> $ in the expression above, we bound it by the sum of
\begin{equation}\label{caracterizacao07}
\limsup_{N \to \infty}\QQ_N \Bigg[
\sup_{0 \leq t \leq T}|M^N_t(H)| > \delta/32
\Bigg]
\end{equation}
and 
\begin{multline}\label{caracterizacao09a}
\limsup_{N \to \infty}\PP_{\mu_{N}} \Bigg[
\eta_.:
\sup_{0 \leq t \leq T}\Big|
\int_{0}^{t} N^2L_N\< \pi^N_{s} , H_{s} \> \,ds - \int_{0}^{t} \<\pi^N_{s} , \Delta H_{s} \> \,ds \\ 
-\int_0^t\eta^{\varepsilon N}_s(1)[\partial_u H_s(0)-cH_s(0)]\,ds
+\int_0^t\eta^{\varepsilon N}_s(N\!-\!1)[\partial_u H_s(1)+cH_s(1)]\,ds \\
-c\int_0^t [\alpha H_s(0)+\beta H_s(1)]\,ds
\Big|>\delta/32
\Bigg],
\end{multline}
where $M^N_t(H)$ was defined in \eqref{M}, $\pi^N_{s}$ is the empirical measure and $\eta^{\varepsilon N}_s(1)$ ($\eta^{\varepsilon N}_s(N\!-\!1)$) is the mean of the box in the right (left) of site $x=1$ ($x=N\!-\!1$). 
 The remark in Section \ref{s4}  (see \eqref{limprob}) helps us to conclude that \eqref{caracterizacao07} is zero. 
Recalling \eqref{gerador_interno_eq} and  \eqref{gerador_fronteira_eq}, the expression $N^2 L_{N}\< \pi_s ,H_s \> $ can be rewritten as 
\begin{equation*}
\begin{split}
N^2 L_{N}\< \pi_s ,H_s \> = & \frac{1}{N}\sum_{x=2}^{N\!-\!2}\eta_s(x)\Delta_N
H_s\left( \pfrac{x}{N}\right) 
+\,\eta_s(1)\Big[\nabla_N^+ H_s(\pfrac{1}{N})-cH_s(\pfrac{1}{N})\Big]\\
&-\eta_s(N\!-\!1)\Big[\nabla_N H_s^-(\pfrac{N\!-\!1}{N}))+cH_s(\pfrac{N\!-\!1}{N})\Big] + c\Big(\alpha H_s(\pfrac{1}{N})+ \beta H_s(\pfrac{N\!-\!1}{N})\Big).
 \end{split}
 \end{equation*}
 Finally it is possible to bound  the expression inside the probability in \eqref{caracterizacao09a} by the sum of the following terms
\begin{equation}\label{final1}
\sup_{0\leq t\leq T}
	\Big|
		\int_0^t\Big\{\pfrac{1}{N}\sum_{x=2}^{N\!-\!2}\eta_s(x)\,\Delta_N
H_s\left( \pfrac{x}{N}\right)
	- \pfrac{1}{N}\sum_{x=1}^{N\!-\!1}\eta_s(x)\,\Delta
H_s\left( \pfrac{x}{N}\right)\Big\}\,ds
	\Big|,
\end{equation}
\begin{equation*}
\sup_{0\leq t\leq T}
	\Big|\int_0^t\big[\eta^{\varepsilon N}_s(1)-\eta_s(1)\big]\,\big[\partial_u H_s(0)-cH_s(0)\big]\,ds\Big|,
\end{equation*}
\begin{equation}\label{final21}
\sup_{0\leq t\leq T}
	\Big|\int_0^t\eta_s(1)\Big\{\big[\nabla_N^+ H_s(\pfrac{1}{N})-cH_s(\pfrac{1}{N})\big]-\big[\partial_u H_s(0)-cH_s(0))\big]\Big\}\,ds\Big|,
\end{equation}
\begin{equation*}
\sup_{0\leq t\leq T}
	\Big|\int_0^t\big[\eta^{\varepsilon N}_s(N\!-\!1)-\eta_s(N\!-\!1)\big]\,\big[\partial_u H_s(1)-cH_s(1)\big]\,ds\Big|,
\end{equation*}
\begin{equation}\label{final31}
\sup_{0\leq t\leq T}
	\Big|\int_0^t\eta_s(N\!-\!1)\Big\{\big[\nabla_N^- H_s(\pfrac{N\!-\!1}{N})-cH_s(\pfrac{N\!-\!1}{N})\big]-\big[\partial_u H_s(1)-cH_s(1))\big]\Big\}\,ds\Big|,
\end{equation}
and
\begin{equation}\label{final4}
\sup_{0\leq t\leq T}
	\Big|c\int_0^t\Big\{ \alpha\big(H_s(\pfrac{1}{N})-H_s(0)\big)+\beta \big(H_s(\pfrac{N\!-\!1}{N})-H_s(1)\big)\Big\}\,ds\Big|.
\end{equation}
As $H\in C^{1,3}([0,T]\times[0,1])$, then it is easy to see that \eqref{final1}, \eqref{final21}, \eqref{final31} an \eqref{final4} go to zero, uniformly in $\eta_.$ and $N$. 
 Thus, instead of \eqref{caracterizacao09a}, we may only look at the following limit,  for $x\in\{1,N\!-\!1\}$ and $\tilde{\delta}>0$, 
\begin{equation*}
\limsup_{N \to \infty}\PP_{\mu_N}\Big[\eta_.:\sup_{0\leq t\leq T}
	\Big|\int_0^t\big[\eta^{\varepsilon N}_s(x)-\eta_s(x)\big]\,\big[\partial_u H_s(u_x)-cH_s(u_x)\,ds\Big|>\tilde{\delta}\Big],
\end{equation*}
where $u_1=0$ and $u_{N\!-\!1}=1$.
Applying  Lemma \ref{substituicao} (Replacement Lemma) we conclude  observing that, when $\varepsilon\to 0$, the above limit goes to 0.
This concludes the proof.
\end{proof}

\subsection{Characterization of Limit Points for $\theta\in [0,1)$}

  As in the last subsection, we will look at the limit points of the sequence $\{\QQ_{N}\}_{N\in\NN}$.
\begin{prop}\label{caracterizacao_master0} If $\QQ^*$ is a limit point of $\{\QQ_{N}\}_{N\in\NN}$, we have
\begin{multline*}
\QQ^*
\Bigg[ \pi_.:\< \rho_{t}, H_{t}\>-\< \gamma , H_{0}\> -\int_{0}^{t} \< \rho_{s} , (\partial_{s}+\Delta)  H_{s} \> \,ds
+ \int_{0}^{t}\big\{\beta\partial_{u}H_{s}(1)-\alpha\partial_{u}H_{s}(0)\big\}\,ds
=0, \\
\;\;
\forall t\in[0,T], 
\;\;
\forall H\in C^{1,2}_0([0,T]\times [0,1])
\Bigg]=1.
\end{multline*}
\end{prop}
\begin{proof}
  By a density argument, it is enough  to verify  that, for $ \delta > 0 $ and $H\in C^{1,3}_0([0,T]\times [0,1])$ fixed, 
\begin{multline*}
\QQ^*
\Bigg[
	\pi_.:\sup_{0\leq t \leq T}
		\Big|
			\< \rho_{t}, H_{t}\>-\< \gamma , H_{0}\> -\int_{0}^{t} \< \rho_{s} , (\partial_{s}+\Delta)  H_{s} \>\, ds
+ \int_{0}^{t}\big\{\beta\partial_{u}H_{s}(1)-\alpha\partial_{u}H_{s}(0)\big\}\,ds 
		\Big|> \delta
\Bigg]=0.
\end{multline*}
Since the set considered above is an open set, we can use the Portmanteau's Theorem directly and bound  the last probability by
\begin{multline*}
\liminf_{N\to\infty}\QQ_{N}
\Bigg[
	\pi_.:\sup_{0\leq t \leq T}
		\Big|
			\< \pi_{t}, H_{t}\>-\< \pi_0 , H_{0}\> -\int_{0}^{t} \< \pi_{s} , (\partial_{s}+\Delta)  H_{s} \> ds
+ \int_{0}^{t}\big\{\beta\partial_{u}H_{s}(1)-\alpha\partial_{u}H_{s}(0)\big\}ds 
		\Big|> \delta
\Bigg].
\end{multline*}
Following the same steps of the last subsection, we need only to bound
\begin{multline}\label{caracterizacao143}
\limsup_{N\to\infty}\PP_{\mu_N}
\Bigg[
	\eta_.:\sup_{0\leq t \leq T}
		\Big|
			\int_{0}^{t} N^2L_N\< \pi^N_{s} , H_{s} \> \,ds -\int_{0}^{t} \< \pi^N_{s}, \Delta H_{s} \> ds
+ \int_{0}^{t}\big\{\beta\partial_{u}H_{s}(1)-\alpha\partial_{u}H_{s}(0)\big\}ds 
		\Big|> \delta/2
\Bigg],
\end{multline}
where $\pi^N_{s}$ is the empirical measure.
As in the case $\theta=1$, it is possible to bound  the expression inside the probability above by the sum of the following terms
\begin{equation*}
\sup_{0\leq t\leq T}
	\Big|
		\int_0^t\Big\{\pfrac{1}{N}\sum_{x=2}^{N\!-\!2}\eta_s(x)\,\Delta_N
H_s\left( \pfrac{x}{N}\right)
	- \pfrac{1}{N}\sum_{x=1}^{N\!-\!1}\eta_s(x)\,\Delta
H_s\left( \pfrac{x}{N}\right)\Big\}\,ds
	\Big|,
\end{equation*}
\begin{equation*}
\sup_{0\leq t\leq T}
	\Big|\int_0^t\eta_s(1)\big[\nabla_N^+ H_s(\pfrac{1}{N})-\partial_u H_s(0)]\,ds\Big|,
	\end{equation*}
	\begin{equation}\label{charc111}
\sup_{0\leq t\leq T}
	\Big|\int_0^t\partial_u H_s(0)\big[\eta_s(1)-\alpha\big]\,ds\Big|,
\end{equation}
\begin{equation*}
\sup_{0\leq t\leq T}
	\Big|\int_0^t\eta_s(N\!-\!1)\big[\nabla_N^- H_s(\pfrac{N\!-\!1}{N})-\partial_u H_s(1)]\,ds\Big|,
	\end{equation*}
	\begin{equation}\label{charc222}
\sup_{0\leq t\leq T}
	\Big|\int_0^t\partial_u H_s(1)\big[\eta_s(N\!-\!1)-\beta\big]\,ds\Big|,
\end{equation}
and
\begin{equation*}
\sup_{0\leq t\leq T}
	\Big|\int_0^t c\Big\{ N^{1-\theta}\Big(\alpha-\eta_s(1)\Big) H_s(\pfrac{1}{N})+ N^{1-\theta}\Big(\beta -\eta_s(N\!-\!1)\Big)H_s(\pfrac{N\!-\!1}{N}) \Big\}\,ds\Big|.
\end{equation*}
  Since $H\in C^{1,3}_0([0,T]\times[0,1])$, and keeping in mind that $\eta_s(x)$ only takes values in $\{0,1\}$, it is easy to see that the terms above (except the terms \eqref{charc111} and \eqref{charc222}) go to zero, uniformly in $\eta_.$. 
 Then, instead of \eqref{caracterizacao143}, we may only look at the following, for all $\tilde{\delta}>0$ and $x\in\{1,N-1\}$, 
\begin{equation*}
\limsup_{N \to \infty}\PP_{\mu_N}\Big[\eta_.:\sup_{0\leq t\leq T}
	\Big|\int_0^t\partial_u H_s(u_x)\big[\eta_s(x)-\zeta_x\big]\,ds\Big|>\tilde{\delta}\Big],
\end{equation*}
where $\zeta_1=\alpha$, $u_1=0$, $\zeta_{N-1}=\beta$ and $u_{N\!-\!1}=1$.  
Applying  Lemma \ref{replace1} (Replacement Lemma), we conclude that  the expressions above is equal to zero
and it finishes this proof.

 \end{proof}

\subsection{Characterization of Limit Points for $\theta\in (1,\infty)$}

  As before we will look at the limit points of the sequence $\{\QQ_{N}\}_{N\in\NN}$.
\begin{prop}\label{caracterizacao_master10} If $\QQ^*$ is a limit point of $\{\QQ_{N}\}_{N\in\NN}$, it is true that
\begin{multline*}
\QQ^*
\Bigg[ \pi_.:
\< \rho_{t}, H_{t}\>-\< \gamma , H_{0}\> - \int_{0}^{t} \< \rho_{s} , (\partial_{s}+\Delta)  H_{s} \> ds+ \int_{0}^{t}\big\{\rho_s(1)\partial_{u}H_{s}(1)-\rho_{s}(0)\partial_{u}H_{s}(0)\big\}ds
=0, \\
\;\;
\forall t\in[0,T], 
\;\;
\forall H\in C^{1,2}_0([0,T]\times [0,1])
\Bigg]=1.
\end{multline*}
\end{prop}
\begin{proof}
  Reasoning in the same way as above, it is enough  to verify  that, for $ \delta > 0 $ and $H\in C^{1,3}([0,T]\times [0,1])$ fixed, we have
\begin{multline*}
\QQ^*
\Bigg[
	\pi_.:\sup_{0\leq t \leq T}
		\Big|
		\< \rho_{t}, H_{t}\>-\< \gamma , H_{0}\> - \int_{0}^{t} \< \rho_{s} , (\partial_{s}+\Delta)  H_{s} \> ds+ \int_{0}^{t}\big\{\rho_s(1)\partial_{u}H_{s}(1)-\rho_{s}(0)\partial_{u}H_{s}(0)\big\}ds
		\Big|> \delta
\Bigg]=0.
\end{multline*}
Here, we have boundary terms in $\rho$.  To be able to use   Portmanteau's Theorem we need to do the same as in Subsection \ref{limit points critical beta}, where we changed the boundary terms $\rho_{s}(0)$ ($\rho_{s}(1)$) by 
$\eta^{\epsilon N}_{s}(1)$ ($\eta^{\epsilon N}_{s}(N\!-\!1)$).  Then, we sum and 
 subtract $ \int_0^t N^2 L_N \< \pi^N_s, H_s \> $ and use the limit \eqref{limprob}, it is enough to analyze
\begin{multline*}
\limsup_{N\to\infty}\PP_{\mu_N}
\Bigg[
	\eta_.:\sup_{0\leq t \leq T}
		\Big|
			\int_{0}^{t} N^2L_N\< \pi^N_{s} , H_{s} \> \,ds - \int_{0}^{t} \< \pi^N_{s}  , \Delta H_{s} \> ds\\+ \int_{0}^{t}\big\{\eta^{\epsilon N}_{s}(N\!-\!1)\partial_{u}H_{s}(1)-\eta^{\epsilon N}_{s}(1)\partial_{u}H_{s}(0)\big\}ds
		\Big|> \delta
\Bigg].
\end{multline*}
Doing the same as in the other cases and using that $H\in C^{1,3}([0,T]\times[0,1])$ and $\theta\in  (1,\infty)$, we just have to analyze
the following limit,  for  all  $\tilde{\delta}>0$ and $x\in\{1,N\!-\!1\}$,
\begin{equation*}
\limsup_{N \to \infty}\PP_{\mu_N}\Big[\eta_.:\sup_{0\leq t\leq T}
	\Big|\int_0^t\big[\eta^{\varepsilon N}_s(x)-\eta_s(x)\big]\,\partial_u H_s(u_x)\,ds\Big|>\tilde{\delta}\Big],
\end{equation*}
where $u_1=0$ and $u_{N\!-\!1}=1$. Applying  Lemma \ref{substituicao} (Replacement Lemma) we conclude that, taking  limit when $\varepsilon\to 0$, the above limit goes to 0.
This concludes the proof of this proposition.
\end{proof}

\section{Uniqueness}\label{s7}
The uniqueness of weak solutions of \eqref{dirichlet} is standard and we refer to \cite{kl} for a proof. For  the equation \eqref{neumann}, a proof of uniqueness can be found in \cite{fgn1}. It remains to prove uniqueness of weak solutions of the parabolic differential equation \eqref{robin}, and we will do that in the next subsection.

\subsection{Uniqueness of weak solutions of \eqref{robin}}

  Now we head to our last statement: the weak solution of \eqref{robin} is unique. To prove this, it is enough to consider  $\rho_0\equiv 0$, $ \alpha = \beta = 0 $  and that our solution $ \rho $ satisfies
\begin{equation}\label{solucao_fraca2}
\begin{split}
\< \rho_{t}, H_{t}\> & = \int_{0}^{t} \< \rho_{s} , (\partial_{s}+\Delta)  H_{s} \> ds\\  
& + \int_{0}^{t}\big\{\rho_s(0)\partial_{u}H_{s}(0)-\rho_{s}(1)\partial_{u}H_{s}(1)\big\}ds  -c\int_{0}^{t}\Big\{H_{s}(0)\rho_{s}(0)+ H_{s}(1)\rho_{s}(1)\Big\}ds,
\end{split}
\end{equation}
for all $ H \in C^{1,2}([0, T] \times [0,1]) $.

  We begin by considering the set $ \mathcal{H} $ of functions $ H \in C^{1}[0,1] $ such that
\begin{enumerate}
\item $ \partial_u H $ is an absolutely continuous function;
\item $ \Delta H \in L^{2}[0,1] $;
\item $ \partial_{u}H(0)=cH(0) $;
\item $ \partial_{u}H(1)=-cH(1) $.
\end{enumerate}

  Observe that the operator $ \Delta: \mathcal{H} \rightarrow L^{2}[0,1] $ is injective in its domain. Given $ g \in L^{2}[0,1]$, we define
\begin{equation*}
(-\Delta)^{-1}g(u):=\int_{0}^{1}G(r,u)g(r)dr,
\end{equation*}
where
\begin{equation*}
G(r,u):=(1+cu)\left(\frac{\sfrac{1}{c}+1-r}{2+c}\right)-(u-r)\charf{[0,u]}(r).
\end{equation*}
It is easy to see that the following properties hold for $ g, h \in L^{2}[0,1] $:
\begin{enumerate}
\item $ (-\Delta)^{-1}g \in \mathcal{H} $;
\item $ (-\Delta)(-\Delta)^{-1}g =g $;
\item $ \<(-\Delta)^{-1}g,h\>=\<g,(-\Delta)^{-1}h\> $ and $ \<(-\Delta)^{-1}g,g\> \geq 0 $.
\end{enumerate}

\begin{prop}\label{prop_unicidade_master}
Let $ \rho $ be a  solution of \eqref{solucao_fraca2}. For any $ t \in [0,T] $ we have
\begin{displaymath}
\< \rho_{t}, (-\Delta)^{-1}\rho_{t}\>=-2\int_{0}^{t}\< \rho_{s}, \rho_{s}\>\, ds.
\end{displaymath}
\end{prop}

This proposition implies the uniqueness of weak solutions. And its proof is similar to the proof of Proposition 3.4 in \cite{fgn2}, if we use the next lemma in the place of  Lemma 3.3 of \cite{fgn2}.

\begin{lema}\label{lema_unicidade}
Let $ H \in \mathcal{H} $ and $ \rho $ a weak solution of \eqref{solucao_fraca2}. We have
\begin{displaymath}
\<\rho_{t}, H\>= \int_{0}^{t}\<\rho_{s}, \Delta H \> \, ds.
\end{displaymath}
\end{lema}

\begin{proof}
  Take $ \{g_{n}\} \in C[0,1] $ a sequence that converges in $L^2$ to $ \Delta H  \in L^{2}[0,1]$. Define the functions
\begin{displaymath}
G_{n}(u)=H(0)+u\partial_{u}H(0)+\int_{0}^{u}\int_{0}^{v}g_{n}(r) \,dr \,dv,
\end{displaymath}
and observe that
\begin{equation}\label{solucaofraca1}
\begin{split}
\<\rho_{t}, G_{n}\> & =\int_{0}^{t}\<\rho_{s}, g_{n}\> \, ds \\
& +\int_{0}^{t}\Bigg\{\rho_{s}(0)\partial_{u}H(0)-\rho_{s}(1)\Big(\partial_{u}H(0)+\int_{0}^{1}g_{n}(r)\, dr \Big)\Bigg\} \,ds \\
& - c\int_{0}^{t}\Bigg\{H(0)\rho_{s}(0)+\rho_{s}(1)\Big(H(0)+\partial_{u}H(0)+ \int_{0}^{1}\int_{0}^{v}g_{n}(r)\, dr \, dv\Big)\Bigg\}\,ds.
\end{split}
\end{equation}
  Now, as $ g_{n} \rightarrow  \Delta H $ in $L^{2}$, we have
\begin{displaymath}
\partial_{u}H(0)+\int_{0}^{1}g_{n}(r)\, dr \rightarrow \partial_{u}H(1),
\end{displaymath}
and
\begin{displaymath}
H(0)+\partial_{u}H(0)+ \int_{0}^{1}\int_{0}^{v}g_{n}(r)\, dr \, dv \rightarrow H(1).
\end{displaymath}
Since $ H \in \mathcal{H} $, taking the limit in the expression $ \eqref{solucaofraca1} $, as $n \rightarrow \infty $, we obtain
\begin{displaymath}
\<\rho_{t}, H\>  =\int_{0}^{t}\<\rho_{s}, \Delta H\> \, ds.
\end{displaymath}

\end{proof}

\appendix
\section{Feynman-Kac's lemma  without invariant measure}

\begin{lema}[Feynman-Kac's lemma  without invariant measure]\label{FK-sem med inv}
Let $\{X_t,\, t\geq 0\}$ be a Markov process in the countable state space $E$, with infinitesimal generator $L$. Let $ \nu $ a probability measure in $E$ and $V:[0,\infty)\times E\to \mathbb R$ be a bounded function.
Denote $L_t= L+V(t,\cdot)$ and $V_t=V(t,\cdot)$. Define
\begin{equation}\label{gamma}
\Gamma_t=\sup_{\Vert f\Vert_2=1}\Big\{\<V_t,f^2\>_\nu+\<Lf,f\>_\nu\Big\}=
\sup_{\Vert f\Vert_2=1}\Big\{\<L_tf,f\>_\nu\Big\},\mbox{  for all } t\geq 0,
\end{equation}
where $\<\cdot,\cdot\>_\nu$ denotes the inner product in $L^2(E,\nu)$ and $\Vert\cdot\Vert_2=\<\cdot,\cdot\>_\nu^{1/2}$. 
Then
\begin{equation*}
\mathbb E_\nu\Big[e^{\int_0^{t}V(r, X_r)\,dr}\Big]\leq\,\exp\Big\{\int_0^t  \Gamma_s\,ds\Big\}.
\end{equation*}
\end{lema}

\begin{proof}

For a function $V:[0,\infty)\times E\to \mathbb R$, define the non-homogeneous semigroup 
$$ (P_{s,t}^Vf)(x)=\mathbb E_x\Big[e^{\int_0^{t-s}V(s+r, X_r)\,dr}f(X_{t-s})\Big], \mbox{ ~ for all } t\geq s\geq 0.$$
Then, $\mathbb E_\nu\Big[e^{\int_0^{t}V(r, X_r)\,dr}\Big]=\<P_{0,t}^{V}1,1\>_{\nu}$.
 To bound $\<P_{0,t}^{V}1,1\>_{\nu}$, we start with the Cauchy-Schwarz inequality
\begin{equation*}
\big\<P_{0,t}^V1,1\big\>_\nu\leq \big\<P_{0,t}^V1,P_{0,t}^V1\big\>_\nu^{1/2}.
\end{equation*}
In the remaining of the proof we will look at $\big\<P_{s,t}^V1,P_{s,t}^V1\big\>_\nu$ as a function of $ s $ and apply Gronwall's inequality.
First of all, notice that
\begin{equation}\label{aaa}
\pfrac{d}{ds} \big\<P_{s,t}^V1,P_{s,t}^V1\big\>_\nu=-2\,\big\<L_s P_{s,t}^V1,P_{s,t}^V1\big\>_\nu.
\end{equation} To show this, we differentiate under the integral sign $$ \big\<P_{s,t}^V1,P_{s,t}^V1\big\>_\nu= \int_E (P_{s,t}^V1)^2(x)\,d\nu(x),$$ and use the equalities $\frac{d}{ds}(P_{s,t}^V1)^2(x)=2(P_{s,t}^V1)(x)\frac{d}{ds}(P_{s,t}^V1)(x)$ and $ \partial_s P_{s,t}f = -L_sP_{s,t}f $ (see \cite{kl}, page 335).
Since $g(x)=(P_{s,t}^V1)(x)/\Vert P_{s,t}^V1\Vert_2$ is such that $\Vert g\Vert_2=1$ and \eqref{gamma}, we have 
\begin{equation*}
2\Gamma_s\geq\frac{2\,\big\<L_s P_{s,t}^V1,P_{s,t}^V1\big\>_\nu}{\,\big\< P_{s,t}^V1,P_{s,t}^V1\big\>_\nu}.
\end{equation*}
Plugging into \eqref{aaa},
\begin{equation*}
\pfrac{d}{ds} \big\<P_{s,t}^V1,P_{s,t}^V1\big\>_\nu \geq (-2\Gamma_s) \big\<P_{s,t}^V1,P_{s,t}^V1\big\>_\nu.
\end{equation*}
Applying Gronwall's inequality, we get
\begin{equation*}
\big\<P_{0,t}^V1,P_{0,t}^V1\big\>_\nu\leq \big\<P_{t,t}^V1,P_{t,t}^V1\big\>_\nu\exp\Big\{\int_0^t  2\Gamma_s\,ds\Big\}=\exp\Big\{\int_0^t  2\Gamma_s\,ds\Big\},
\end{equation*}
where the last equality follows from the fact that $P_{t,t}^V1(x)=1$,
and it finishes the proof.
\end{proof}

\section*{Acknowledgements}
The authors A. Neumann and R.R. Souza was partially supported by FAPERGS (proc. 002063-2551/13-0).
A. N. was partially supported through a grant ``L'OR\' EAL - ABC - UNESCO Para Mulheres na Ci\^encia''.

\end{document}